\documentclass[11pt,reqno]{amsart}
\usepackage{amssymb, amsmath, amsthm,amsrefs,marvosym,wasysym}
\usepackage{latexsym}
\usepackage{color}
\usepackage{exscale}
\usepackage{fullpage}
\usepackage{rotating}
\usepackage{bm, bbm, mathrsfs}
\usepackage{graphics, graphicx}
\newtheorem{theorem}{Theorem}[section]
\newtheorem{lemma}[theorem]{Lemma}
\newtheorem{proposition}[theorem]{Proposition}

\newtheorem{definition}{Definition}[section]

\newtheorem{example}[theorem]{Example}
\newtheorem{observation}{Observation}
\newtheorem{remark}{Remark}[section]

\def\eq#1{(\ref{#1})}
\def\nn{\nonumber}
\def\({\left(\begin{array}{cccccc}}
\def\){\end{array}\right)}

\def\eq#1{(\ref{#1})}
\def\nn{\nonumber}

\def\({\left(\begin{array}{cccccc}}
\def\){\end{array}\right)}

\def\bes{\begin{eqnarray}}
\def\ees{\end{eqnarray}}

\newcommand{\del}{\partial}

\newcommand{\beq}{\begin{equation}}
\newcommand{\eeq}{\end{equation}}
\newcommand{\bea}{\begin{eqnarray}}
\newcommand{\eea}{\end{eqnarray}}
\newcommand{\beann}{\begin{eqnarray*}}
\newcommand{\eeann}{\end{eqnarray*}}

\newcommand{\eps}{\ensuremath{\varepsilon}}

\newcommand{\RR}{\mathbb{R}}

\newcommand{\vph}{{\varphi}}

\DeclareMathOperator{\sgn}{sgn}

\newcommand{\bu}{{\bf u}}
\newcommand{\bv}{{\bf v}}
\newcommand{\bF}{{\bf F}}
\newcommand{\bx}{{\bf x}}
\newcommand{\bX}{{\bf X}}
\newcommand{\by}{{\bf y}}
\newcommand{\bY}{{\bf Y}}
\newcommand{\ba}{{\bf a}}
\newcommand{\bb}{{\bf b}}
\newcommand{\bz}{{\bf z}}
\newcommand{\bw}{{\bf w}}

\newcommand{\nquad}{\negthickspace\negthickspace
\negthickspace\negthickspace}
\newcommand{\nqquad}{\nquad\nquad}

\DeclareMathOperator{\grad}{grad}
\DeclareMathOperator{\dv}{div}

\numberwithin{equation}{section}
 
\begin{document}

\title{Cumulative Euler flows}

\begin{abstract}
	We consider the compressible Euler system for ideal gas flow in the absence of 
	any forces except the internal thermodynamic pressure. In this setting, and in dimensions
	higher 1, it is known that wave-focusing can drive Euler solutions to amplitude blowup 
	in finite time from bounded initial data. In the known cases (self-similar, radial flows
	\cites{gud,hun_60,jt3,laz,mrrs1,jls})
	the primary flow variables are standard functions at time of blowup.
	It is natural to ask if the Euler system admits even more singular behavior, and
	specifically whether accumulation of mass, i.e., the appearance of a Dirac delta 
	in the density field, is possible.
	
	We consider the class of radial affine motions \cites{mcvittie,sed, kell,sid_2014}
	which are conveniently obtained via a Lagrangian formulation.
	This class does include examples of cumulative behavior, and we observe that 
	there are two distinct mechanisms for accumulation, due to inertial effects or adverse 
	pressure gradients, respectively. However, we show that all affine cumulative solutions 
	necessarily exhibit unphysical behavior due to initially unbounded velocity and/or 
	acceleration in the far-field. 
	We also analyze the behavior of characteristics in 
	cumulative flows and consider concrete examples, including a class of 1-dimensional, 
	non-affine flows. Finally, we discuss the possibility of modifying the 
	known examples to obtain physically acceptable gas flows displaying 
	accumulation.

\bigskip
	
	\noindent
%{\bf Key words.} Compressible fluid flow, multi-d Euler system, similarity solutions, radial symmetry, unbounded solutions, contact discontinuity

\noindent
%{\bf AMS subject classifications.} 35L45, 35L67, 76N10, 35Q31
\end{abstract}

\author{Helge Kristian Jenssen }\address{H.~K.~Jenssen, Department of
Mathematics, Penn State University,
University Park, State College, PA 16802, USA ({\tt
jenssen@math.psu.edu}).}

\date{\today}
\maketitle

\tableofcontents

%%%%%%%%%%%%%%%%%%%%%%%%%%%%%%%%%%%%%%%%%%%%%%%%
%%%%%%%%%%%%%%%%%%%%%%%%%%%%%%%%%%%%%%%%%%%%%%%%
\section{Introduction}\label{intro} 
%%%%%%%%%%%%%%%%%%%%%%%%%%%%%%%%%%%%%%%%%%%%
%%%%%%%%%%%%%%%%%%%%%%%%%%%%%%%%%%%%%%%%%%%%
We revisit a class of essentially explicit solutions to the radially symmetric Euler system 
describing compressible gas flow in the absence of external forces, self-gravitation, and
electrodynamic effects. This well-studied
class has been considered by a number of authors, starting with independent 
investigations in the early 1950s by McVittie \cite{mcvittie}, Sedov \cite{sed}, and Keller \cite{kell}.
The solutions in question may be characterized by positing either that the velocity field
is linear in the spatial (Eulerian) coordinate, or by insisting that the motion of fluid elements
factors when viewed as a function of time and a Lagrangian variable (see Section \ref{septd_aff},
Remarks \ref{mcvittie_approach} and \ref{keller_approach}).
The aforementioned authors analyzed the resulting flows in some detail. In 
particular, Sedov \cite{sed} and Keller \cite{kell} exhibited concrete cases 
of {\em accumulating} radial flows, i.e., flows in which 
all fluid particles reach a common point (the center of motion) at a common critical 
time $t_c$. (This phenomenon is also referred to as {\em collapse} of the fluid.) 
Cumulative solutions have been considered in models for inertial fusion, 
see \cite{am} for a detailed discussion.

Accumulation represents a highly singular phenomenon in fluid flow as it signals 
the appearance of a Dirac delta in the density field. This is therefore a more severe 
singularity than what occurs in the, by now, well-studied case of amplitude blowup
in imploding self-similar Euler flows \cites{gud,hun_60,jt3,laz,mrrs1,jls} (see Section
\ref{affine_anal} for a brief discussion). For the latter class of solutions no
accumulation occurs, notwithstanding the fact that the density field can suffer 
blowup at the center of motion.\footnote{In the work \cite{sid_2024}  Sideris 
demonstrates the possibility of physically acceptable cumulative solutions for 
motions of a ball of certain hyperelastic materials surrounded by vacuum. 
Another instance of accumulation occurs in self-gravitating gas flow
(gravitational collapse); see \cites{dxy,mak,fu_lin,dlyy,ghj}.}

We are therefore motivated to revisit the class of solutions indicated above. 
Following Keller \cite{kell} and Sideris \cites{sid_2014,sid_2017,sid_2024}, we find it convenient 
to obtain them in a Lagrangian framework. As a minor point about the general 
class of these ``linear velocity'' or ``separated'' solutions, we observe that they can all
be obtained as {\em affine} flows, i.e., flows whose  motions 
are linear in the Lagrangian (``particle'') coordinate; see Proposition \ref{septd_vs_affine}.

We then analyze cumulative affine flows. In this case accumulation is 
necessary total (i.e., {\em all} fluid particles reach the center of motion at the same time; 
cf.\ Section \ref{affine_anal}), and we find that there are two distinct 
types: Accumulation is caused either entirely by inertial effects in a fluid with unbounded 
velocities in the far-field and no pressure gradients, or it is due to an adverse, positive pressure gradient
(i.e., the pressure increases with the distance form the center of motion). Concrete 
examples are provided to illustrate flows of both types. 

On the other hand, we also observe that any (non-trivial) affine flow, cumulative or not, suffers 
from the following un-physical feature: The initial velocity and/or the initial 
acceleration field is necessarily unbounded in the far-field.\footnote{This 
applies in particular to certain 1-dimensional non-accumulating 
flows with vanishing initial velocity 
in which the conserved quantities are arbitrarily small in 
$BV(\RR)$ and $L^1(\RR)$. These flows are globally defined in time and 
are covered by Glimm's theorem \cite{gl}. The fact that the 
velocity (a non-conserved quantity) suffers instantaneous blowup illustrates 
a subtle point about Glimm's theorem; for details see Example \ref{instant_blowup} and Remark 
\ref{instant_blowup_rmk}.}
The upshot is that all cases of accumulation in radial affine flows is 
readily accounted for and not all that surprising.

This leaves open the question of existence of ``physically acceptable'' cumulative flows
of ideal gases.
In particular, it is of interest to consider the possibility of modifying the far-field initial data 
in the known radial accumulating flows in such a manner that accumulation 
still occurs, while avoiding infinite velocities or accelerations in the far-field.

For this it is relevant to consider the characteristics (and especially the 
1-characteristics) in radial affine flows. We therefore consider these for both types of
cumulative behavior. Since 1-characteristics, by definition, approach the center 
of motion faster than particle trajectories, it is immediate that 1-characteristics in 
cumulative flows must reach the center of motion $r=0$ no later than the critical time $t_c$.
(In fact, concrete examples show that there are at least two possibilities: Either 
each 1-characteristic reaches $r=0$ strictly before $t_c$, or all 1-characteristics, 
as well as all particle trajectories, accumulate at $r=0$ at time $t_c$; see Examples
\ref{ex_lam_0} and \ref{ex_lam_neg}). In our view, this circumstance makes it probable that any modification of the 
initial data in accumulating affine flows will influence, and ``pollute,'' the particle trajectories
near $r=0$ that are supposed to produce accumulation, with the outcome that 
cumulative behavior is avoided. For more details see Section \ref{disc}.
Resolving the issue of accumulation from ``physically acceptable'' data, including the 
possibility of other types of cumulative behavior, remains an open problem 
in gas dynamics. (In Example \ref{ex_non_affine} we give a class of 1-dimensional, non-affine, 
accumulating flows; again, however, they all have an initially unbounded velocity field.)

The rest of the article is organized as follows. In the next section we record the 
non-isentropic Euler system and specialize to radial flows of an ideal, polytropic 
gas. Sections \ref{lagr_formln} and \ref{eul_flws} briefly summarizes the Lagrangian 
formulation and derives the Lagrangian equation of motion (EOM) for prescribed 
pressure and density reference profiles. The EOM is specialized to radial flows in 
Section \ref{radial_motions} and to radial separated flows in Section \ref{septd_aff}.
We show that, due to the freedom of choosing reference profiles for the 
thermodynamic fields, the class of separated flows is covered by the seemingly more
special class of affine flows (Proposition \ref{septd_vs_affine}). 
In Section \ref{affine_anal} we focus on the latter class of radial flows,
record the aforementioned far-field behaviors, and briefly discuss the 
general issue of accumulation in Euler flows. Sections \ref{aff_lam_0} and Section
\ref{aff_lam_not_0} provide details on cumulative radial affine flows, 
including expressions for time of accumulation, directions of flow,
properties of the far-field, and behavior of 1-characteristics. 
Concrete examples illustrate the findings, including the possibility of 
instantaneous blowup in the velocity for small $BV$-solutions in dimension 1
(cf. Remark \ref{instant_blowup_rmk}). In Section \ref{non_affine} we 
demonstrate, via an explicit class of 1-dimensional solutions, 
that accumulation can occur also in non-affine flows.
Finally, Section \ref{disc} summarizes our 
findings and discusses the prospect of building cumulative solutions with 
a physically acceptable far-field behavior. In particular, we comment on 
the effects of different (continuous or discontinuous) modifications of the initial 
data in cumulative radial affine flows.

%%%%%%%%%%%%%%%%%%%%%%%%%%%%%%%%%%%%%%%%%%%%%%%%
%%%%%%%%%%%%%%%%%%%%%%%%%%%%%%%%%%%%%%%%%%%%%%%%
\section{The Euler equations for compressible gas dynamics}\label{compr_euler} 
%%%%%%%%%%%%%%%%%%%%%%%%%%%%%%%%%%%%%%%%%%%%
%%%%%%%%%%%%%%%%%%%%%%%%%%%%%%%%%%%%%%%%%%%%
The non-isentropic compressible Euler system expresses conservation of 
mass, linear momentum, and energy in the absence of viscosity and heat conduction: 
\begin{align}
	\rho_t+\dv_{\bx}(\rho \bu)&=0 \label{mass_cons}\\
	(\rho{\bf  u})_t+\dv_{\bx}[\rho {\bu}\otimes{\bu}]+\grad_{\bf  x} p&={\bf 0} \label{mom_cons}\\
	(\rho E)_t+\dv_{\bx}[(\rho E+p){\bu}]&=0.\label{energy_cons}
\end{align}
The independent variables are time $t\in\RR$ and position ${\bf  x}\in\RR^n$, and the 
dependent variables are density $\rho$, fluid velocity ${\bf  u}$, 
and specific internal energy $\eps$; the specific total energy is $E=\eps+\textstyle\frac{1}{2}|{\bf  u}|^2$. 
Recalling the Gibbs relation
\[d\eps=\theta dS+\textstyle\frac{p}{\rho^2}d\rho,\]
where $\theta$ is the absolute temperature and $S$ is the specific
entropy, we have that $\eps=\eps(\rho,S)$ satisfies
\beq\label{eos}
	\eps_S=\theta\qquad\text{and}\qquad \eps_\rho=\textstyle\frac{p}{\rho^2}.
\eeq
Introducing the convective derivative 
\[\textstyle\frac{D}{Dt}\equiv \del_t+(\bu\cdot\nabla_\bx),\]
and performing standard simplifications (assuming $C^1$-smoothness and $\rho>0$), 
the Euler system takes the form
\begin{align}
	\textstyle\frac{D\rho}{Dt}+\rho\dv_{\bf x}\bf u&=0 \label{mass}\\
	\textstyle\frac{D\bu}{Dt}+\frac{1}{\rho}\grad_{\bx} p&={\bf 0}\qquad\label{mom}\\
	\textstyle\frac{D\eps}{Dt}+\frac{p}{\rho}\dv_{\bx}\bu&=0.\label{int_energy}
\end{align}
Using \eq{eos} and \eq{mass}, and assuming also $\theta>0$, \eq{int_energy} reduces to 
\beq\label{S_cons}
	\textstyle\frac{DS}{Dt}=0.
\eeq
Alternatively, regarding the pressure as a function of $\rho$ and $\eps$, 
we can replace \eq{int_energy} by the equation for the pressure:
\beq\label{p_eqn}
	\textstyle\frac{Dp}{Dt}+(\rho p_\rho+\textstyle\frac{p}{\rho}p_\eps)\dv_{\bx}\bu=0.
\eeq
In what follows we restrict attention to ideal gases, in which case the 
pressure $p$ given by
\beq\label{pressure1}
	p(\rho,\eps)=(\gamma-1)\rho \eps\qquad\qquad\text{($\gamma>1$ constant).}
\eeq
We also assume that the gas is polytropic: 
\beq\label{polytr}
	\eps=c_v\theta \qquad\qquad\text{($c_v>0$ constant).}
\eeq
The equation for $\eps$ then takes the form
\beq\label{int_energy_ideal}
	\textstyle\frac{D\eps}{Dt}+(\gamma-1)\eps\dv_{\bx}\bu=0,
\eeq
while the pressure equation \eq{p_eqn} takes the form
\beq\label{p_eqn_ideal}
	\textstyle\frac{Dp}{Dt}+\gamma p\dv_{\bx}\bu=0.
\eeq
In ideal polytropic gases the
internal energy is proportional to the temperature $\theta$ of the gas: $e\propto\theta$.
For later reference we record that the local sound speed $c$ is given by
\beq\label{sound_speed}
	c=\sqrt{\textstyle\frac{\gamma p}{\rho}}=\sqrt{\gamma(\gamma-1)\eps}\propto\sqrt\theta.
\eeq
%We refer to $\rho$, $\bf u$, $p$, $\eps$, $c$, $\theta$ as primary (undifferentiated) 
%flow variables.

%%%%%%%%%%%%%%%%%%%%%%%%%%%%%%%%%%%%
\subsection{Radial Euler flows in Eulerian frame}\label{rad_flows}
%%%%%%%%%%%%%%%%%%%%%%%%%%%%%%%%%%%%
In {\em radial flows} the flow variables depend on
position only through $r=|{\bf x}|$, and the 
velocity field is purely radial, viz.\ ${\bf u}(t,\bx)=u(t,r)\frac{\bf x}{r}$. 
With these assumptions, and within $C^1$-regions of the flow, 
\eq{mass} and \eq{mom} become
\begin{align}
	\rho_t+u\rho_r+\rho(u_r+\textstyle\frac{(n-1)u}{r}) &= 0\label{mass_rad}\\
	u_t+ uu_r +\textstyle\frac{1}{\rho}p_r&= 0,\label{mom_rad}
\end{align}
respectively, where $\rho=\rho(t,r)$, $u=u(t,r)$, and $p=p(t,r)$. 
The forms \eq{S_cons}, \eq{int_energy_ideal}, \eq{p_eqn_ideal} of the 
energy equation take the forms
\begin{align}
	S_t+uS_r&=0\label{S_rad}\\
	\eps_t+u\eps_r+(\gamma-1)\eps(u_r+\textstyle\frac{(n-1)u}{r})&=0
	\label{int_energy_ideal_rad}\\
	p_t+up_r+\gamma p(u_r+\textstyle\frac{(n-1)u}{r})&=0,\label{p_eqn_ideal_rad}
\end{align}
respectively.
%Alternatively, in terms of the sound speed $c:=\sqrt{\frac{\gamma p}{\rho}}$, the energy equation 
%takes the form 
%\beq\label{c_rad}
%	c_t+uc_r+{\textstyle\frac{\gamma-1}{2}}c(u_r+\textstyle\frac{(n-1)u}{r})=0.
%\eeq
%%%%%%%%%%%%%%%%%%%%%%%%%%%%%%%%%%%%
\begin{definition}
	Any solution to \eq{mass_rad}-\eq{mom_rad} and one of \eq{S_rad}-\eq{p_eqn_ideal_rad}
	is referred to as a radial Euler flow.
\end{definition}
%%%%%%%%%%%%%%%%%%%%%%%%%%%%%%%%%%%%

%%%%%%%%%%%%%%%%%%%%%%%%%%%%%%%%%%%%
%%%%%%%%%%%%%%%%%%%%%%%%%%%%%%%%%%%%
\section{Lagrangian formulation}\label{lagr_formln}
%%%%%%%%%%%%%%%%%%%%%%%%%%%%%%%%%%%%
%%%%%%%%%%%%%%%%%%%%%%%%%%%%%%%%%%%%
In this section we briefly recall some standard definitions, introduce notation,
and formulate a lemma for later use. In what follows, gradients of scalar 
quantities are considered as row vectors; all other vectors are 
considered as column vectors.

In the Lagrangian formulation one considers the flow variables as they evolve 
for fixed ``particles.'' We fix a reference configuration labeled by $\by\in\RR^n$,
$\by$ denoting the Lagrangian coordinate.
The `$\by$' here may denote the initial position in physical space of ``particle $\by$,'' 
but that is not necessarily the case. (E.g., for radial flows one may let 
$\by$, or rather $|\by|$, be a ``Lagrangian mass coordinate,'' cf.\ Remark \ref{keller_approach}.) A generic point 
in physical space is labeled $\bx\in\RR^n$ (Eulerian frame); functions of $(t,\bx)$ (or of $(t,\by)$)
are referred to as Eulerian (Lagrangian, respectively) functions. 

A {\em motion} is a map $\bx=\bX(t,\by)$ which gives the position in physical space 
at time $t$ of the particle indexed by $\by$. At this stage it is not 
assumed that the motion $\bX(t,\by)$ gives a physical flow described by 
the Euler (or any other) system of field equations. 

In what follows, the motion $\bX(t,\by)$ is assumed to be $C^1$  and globally 
$C^1$-invertible at each time $t$. The inverse, position-to-label map is denoted 
$\bY(t,\bx)$, i.e.,
\beq\label{inverse}
	\bY(t,\bX(t,\by))=\by.
\eeq
Differentiating with respect to $\by$ gives
\beq\label{jacs}
	\big(D_\bx\bY(t,\bx)\big)\big|_{\bx=\bX(t,\by)}=(D_\by\bX(t,\by))^{-1},
\eeq
where $D_\bz$ denotes the Jacobian with respect to the variable $\bz$.
The motion $\bX(t,\by)$ defines a velocity field $\bu(t,\bx)$ in physical space according to
\beq\label{u1}
	\bu(t,\bX(t,\by)):=\dot\bX(t,\by),
\eeq
or, equivalently,
\beq\label{u2}
	\bu(t,\bx)=\dot\bX(t,\by)\big|_{\by=\bY(t,\bx)},
\eeq
where we use a dot for time differentiation of Lagrangian functions.
The convective derivative associated with the motion $\bX(t,\by)$ of an
Eulerian scalar field $\vph(t,\bx)$ is given by
\[\frac{D\vph}{Dt}(t,\bx):=\vph_t(t,\bx)+(\bu(t,\bx)\cdot\nabla)\vph(t,\bx)
\equiv\vph_t(t,\bx)+\nabla_\bx\vph(t,\bx)\cdot \bu(t,\bx),\]
where $\nabla_\bx$ is  the gradient with respect to $\bx$ and $t$-subscript 
denotes time differentiation of Eulerian functions.
Similarly, the convective derivative of an Eulerian vector field
$\bv(t,\bx)\in\RR^{n\times1}$ is given by
\[\textstyle\frac{D\bv}{Dt}(t,\bx):=\bv_t(t,\bx)+(\bu(t,\bx)\cdot\nabla)\bv(t,\bx)
\equiv\bv_t(t,\bx)+D_\bx\bv(t,\bx)\bu(t,\bx).\]
Convective differentiation corresponds to time 
differentiation along particle trajectories, i.e., with
\beq\label{lagr_vf}
	\bw(t,\by):=\bv(t,\bX(t,\by)),
\eeq
we have
\beq\label{conv_lagr}
	\dot\bw(t,\by)=\textstyle\frac{D\bv}{Dt}\big|_{(t,\bX(t,\by))}.
\eeq
Finally, we define the Jacobian determinant
\beq\label{J}
	J(t,\by):=\det D_\by\bX(t,\by),
\eeq
and record the standard result \cite{serrin} that
\beq\label{J_deriv}
	\dot J(t,\by)=J(t,\by)(\dv_\bx\bu)\big|_{(t,\bX(t,\by))}.
\eeq
The following lemma shows how $J(t,\by)$ can be used to modulate a Lagrangian scalar field 
to obtain solutions to certain transport equations in the Eulerian frame; in the next section 
this is applied to the thermodynamic fields in Euler flow.
%%%%%%%%%%%%%%%%%%%%%%%%%%%%%%%%%%%%%%%%%
\begin{lemma}\label{scalar_pde}
	Let $\bX(t,\by)$ be a given (sufficiently smooth) motion with associated velocity field 
	$\bu(t,\bx)$ defined by \eq{u1}.
	Then, given the Lagrangian scalar field $\bar\vph(\by)$, the Eulerian scalar field 
	$\vph(t,\bx)$ defined by
	\beq\label{phi}
		\vph(t,\bX(t,\by)):=J(t,\by)^{-\alpha}\bar\vph(\by)\qquad\text{($\alpha$ constant)}
	\eeq
	satisfies the PDE
	\beq\label{pde}
		\textstyle\frac{D\vph}{Dt}(t,\bx)+\alpha(\vph\dv \bu)(t,\bx)=0.
	\eeq
\end{lemma}
%%%%%%%%%%%%%%%%%%%%%%%%%%%%%%%%%%%%%%%%%
\begin{proof}
	Differentiating \eq{phi} with respect to $t$ gives
	\[\vph_t\big|_{(t,\bX(t,\by))}+\nabla_\bx\vph\big|_{(t,\bX(t,\by))}\cdot \dot\bX(t,\by)
	=-\alpha J(t,\by)^{-\alpha-1}J_t(t,\by)\bar\vph(\by),\]
	or, according to \eq{u1} and \eq{J_deriv},
	\[\vph_t\big|_{(t,\bX(t,\by))}+(\bu\cdot\nabla_\bx\vph)\big|_{(t,\bX(t,\by))}
	=-\alpha J(t,\by)^{-\alpha}\bar\vph(\by)(\dv_\bx\bu)\big|_{(t,\bX(t,\by))},\]
	or, according to the definition of $\vph$,
	\[\vph_t\big|_{(t,\bX(t,\by))}+(\bu\cdot\nabla_\bx\vph)\big|_{(t,\bX(t,\by))}
	+\alpha(\vph\dv_\bx\bu)\big|_{(t,\bX(t,\by))}=0,\]
	which is \eq{pde} at the point $(t,\bX(t,\by))$.
\end{proof}

%%%%%%%%%%%%%%%%%%%%%%%%%%%%%%%%%%%%
%%%%%%%%%%%%%%%%%%%%%%%%%%%%%%%%%%%%
\section{Euler flows and the equation of motion}\label{eul_flws}
%%%%%%%%%%%%%%%%%%%%%%%%%%%%%%%%%%%%
%%%%%%%%%%%%%%%%%%%%%%%%%%%%%%%%%%%%
We now consider the case where the motion $\bX(t,\by)$ records the 
particle trajectories in a gas flow described by the Euler system. The gas
is assumed to ideal and polytropic. In this situation Lemma \ref{scalar_pde} 
shows that, given the motion $\bX(t,\by)$, we can immediately ``solve'' the equations 
for density, specific internal energy, pressure, and specific entropy. This is reasonable: 
knowing how each particle moves, it should be possible to determine 
the thermodynamic fields in terms of their reference profiles.

More precisely, with $\bX(t,\by)$, $\bu(t,\bx)$, and $J(t,\by)$ as above, and 
for given reference profiles $\bar\rho(\by)$, $\bar \eps(\by)$, 
$\bar p(\by)$, $\bar S(\by)$, Lemma \ref{scalar_pde} shows that 
the fields $\rho(t,\bx)$,  $\eps(t,\bx)$, $p(t,\bx)$,  $S(t,\bx)$ defined by
\begin{align}
	\rho(t,\bX(t,\by))&:=J(t,\by)^{-1}\bar\rho(\by)&\qquad\text{($\alpha=1$)}\label{rho_formula}\\
	\eps(t,\bX(t,\by))&:=J(t,\by)^{1-\gamma}\bar\eps(\by)&\qquad\text{($\alpha=\gamma-1$)}
	\label{eps_formula}\\
	p(t,\bX(t,\by))&:=J(t,\by)^{-\gamma}\bar p(\by)&\qquad\text{($\alpha=\gamma$)}
	\label{p_formula}\\
	S(t,\bX(t,\by))&:=\bar S(\by)&\qquad\text{($\alpha=0$)}\label{S_formula}
\end{align}
solve \eq{mass}, \eq{int_energy_ideal}, \eq{p_eqn_ideal}, \eq{S_cons}, respectively.
(For density and specific entropy this holds for any gas, ideal or not.)

Of course, this ``solving'' is illusory: The motion $\bX(t,\by)$, and hence $J(t,\by)$, 
are not known and must be determined from the momentum equation \eq{mom}, 
which in turn involves the density and pressure fields.
Rather, the point is that by working in the Lagrangian setting, 
Lemma \ref{scalar_pde} allows us to write down a single, closed equation for 
the motion. 

Specifically, applying \eq{conv_lagr} with $\bv=\bu=\dot \bX$, \eq{mom} 
gives the equation of motion for particle trajectories
\beq\label{eom_lagr}
	\nqquad\nqquad\nqquad\nqquad
	\text{(EOM)}\qquad\qquad\qquad\ddot\bX(t,\by)=-\textstyle\frac{(\nabla_\bx p)^T}{\rho}\big|_{(t,\bX(t,\by))}.
\eeq
For the case of an ideal gas we can use Lemma \ref{scalar_pde} (i.e., \eq{rho_formula}
and \eq{p_formula}) to write out
the right-hand side of \eq{eom_lagr} explicitly in terms of the motion and the reference profiles of 
density and pressure. We thus obtain a single second order
equation for the motion. In general this is a  complicated, nonlinear PDE for $\bX(t,\by)$.
However, under certain conditions, e.g., constraints on the geometry of particle trajectories,
 this PDE  can be analyzed in detail.

%%%%%%%%%%%%%%%%%%%%%%%%%%%%%%
\subsection{Lagrangian EOM with prescribed density and pressure}
\label{rho_and_p}
%%%%%%%%%%%%%%%%%%%%%%%%%%%%%%
In this case we are given reference profiles $\bar\rho(\by)$ and 
$\bar p(\by)$ for the density and pressure. To write out the right-hand 
side of \eq{eom_lagr} we first use  \eq{p_formula} to express the pressure as
\[p(t,\bx)=\big(\textstyle\frac{\bar p(\by)}{J(t,\by)^\gamma}\big)\big|_{\by=\bY(t,\bx)}.\]
Making use of \eq{inverse} and \eq{jacs}, we obtain
\begin{align}
	\nabla_\bx p\big|_{(t,\bX(t,\by))}
	&=\nabla_\bx\big(\textstyle\frac{\bar p(\by)}{J(t,\by)^\gamma}\big|_{\by=\bY(t,\bx)}
	\big)\big|_{\bx=\bX(t,\by)}\nn\\
	&=\big[\nabla_\by\big(\textstyle\frac{\bar p(\by)}{J(t,\by)^\gamma}\big)\big|_{\by
	=\bY(t,\bx)}D_\bx\bY(t,\bx)\big]\big|_{\bx=\bX(t,\by)} \nn\\
	&=\nabla_\by\big(\textstyle\frac{\bar p(\by)}{J(t,\by)^\gamma}\big)
	(D_\by\bX(t,\by))^{-1},
	\label{grad_p_p_given}
\end{align}
Using \eq{grad_p_p_given} together with \eq{rho_formula} in \eq{eom_lagr} 
yields the following EOM for the motion $\bX(t,\by)$:
\beq\label{eom_lagr_rho_p}
	\ddot\bX(t,\by)=-\textstyle\frac{J(t,\by)}{\bar\rho(\by)}(D_\by\bX(t,\by))^{-T}
	\big(\nabla_\by\big(\textstyle\frac{\bar p(\by)}{J(t,\by)^\gamma}\big)\big)^T,
\eeq
where $J(t,\by)$ is given by \eq{J}.
%%%%%%%%%%%%%%%%%%%%%%%%%%%%%%
\begin{remark}\label{Gal_inv}
	Similar calculations yield the EOM in terms of reference profiles for 
	$\rho$ and either $\eps$ or $S$. Also, Galilean invariance is 
	immediate from the form of \eq{eom_lagr_rho_p}: If $\bX(t,\by)$ 
	is a solution of \eq{eom_lagr_rho_p}, then so is $\bX(t,\by)+\bX_0+t\bw$
	for any choice of constant vectors $\bX_0$ and $\bw$.
\end{remark}
%%%%%%%%%%%%%%%%%%%%%%%%%%%%%%

%%%%%%%%%%%%%%%%%%%%%%%%%%%%%%
%%%%%%%%%%%%%%%%%%%%%%%%%%%%%%
\section{Radial motions}\label{radial_motions}
%%%%%%%%%%%%%%%%%%%%%%%%%%%%%%
%%%%%%%%%%%%%%%%%%%%%%%%%%%%%%
In this section we derive the Lagrangian EOM for {\em radial motions}, i.e., motions in 
which particles move along straight, radial lines in $\RR^n$, the motion along such 
lines depending only on time and the distance to the origin. 

It is assumed that the reference density and pressure profiles are radially 
symmetric functions of the Lagrangian coordinate $\by$, and 
with a slight abuse of notation we write
\[\bar p(\by)=\bar p(s)\qquad\bar\rho(\by)=\bar\rho(s)\qquad s=|\by|.\]
Note that with this setup, all particles $\by$ with the same value of $|\by|$ are 
labeled by the same number $s$.
We stress that we are free to choose the Lagrangian coordinate $\by$ (or $s$),
as well as the reference profiles $\bar\rho(s)$ and $\bar p(s)$.
This freedom will be used below to demonstrate that 
two approaches to building concrete solutions (viz.\ separated radial 
motions and affine radial motions) generate the same class of flows
in physical space.

Letting $R(t,s)$ denote the radial position in physical space
at time $t$ of the particle labeled by $s$, we posit
\beq\label{rad_motn}
	\bX(t,\by)=R(t,s)\textstyle\frac{\by}{s}\qquad\qquad (s=|\by|)
\eeq
and the first step is to derive the (radial) EOM for $R(t,s)$. 

For a scalar field $\varphi(\by)$ and vector field $\bF(\by)\in\RR^{n\times 1}$, we have
(recalling the convention that gradients are treated as row vectors)
\[D_\by(\varphi(\by)\bF(\by))=\bF(\by)\nabla_\by\varphi(\by)+\varphi(\by)D_\by F(\by).\]
Also, for a radial scalar field $\varphi(\by)=\varphi(s)$ we have
\[\nabla_\by\varphi(\by)=\textstyle\frac{\varphi'(s)}{s}\by^T.\]
Applying these relations to \eq{rad_motn} we obtain
\begin{align}
	D_\by\bX(t,\by)&=\textstyle\frac{R}{s}
	\left(I+\frac{1}{R}\big(\textstyle\frac{R}{s}\big)_{\!s}\by\by^T\right),
\end{align}
where $I$ denotes the $n\times n$ identity matrix.
Next, making use of the formulae 
\[\det(I+\ba\bb^T)=1+\ba\cdot\bb\qquad(I+\ba\bb^T)^{-1}=I-\textstyle\frac{1}{1+\ba\cdot\bb}\ba\bb^T,\]
for column vectors $\ba$ and $\bb$ (with $\ba\cdot\bb\neq-1$), we obtain
\beq\label{jac}
	J(t,s)\equiv J(t,\by)=\det(D_\by\bX(t,\by))
	=\big(\textstyle\frac{R}{s}\big)^{n}
	\det\left(I+\frac{1}{R}\big(\textstyle\frac{R}{s}\big)_s\by\by^T\right)
	=\big(\textstyle\frac{R}{s}\big)^{n-1}R_s,
\eeq
and (by symmetry of $D_\by\bX(t,\by)$)
\[\left(D_\by\bX(t,\by)\right)^{-T}=\textstyle\frac{s}{R}
\left(I+\frac{1}{R}\big(\textstyle\frac{R}{s}\big)_{\!s}\by\by^T\right)^{-1}
=\textstyle\frac{s}{R}\left(I-\frac{1}{sR_s}\big(\textstyle\frac{R}{s}\big)_{\!s}\by\by^T\right).\]
Using all this in \eq{eom_lagr_rho_p} yields the following radial EOM for $R=R(t,s)$:
\beq\label{eom_rad}
	\ddot R=-\textstyle\frac{1}{\bar\rho(s)}\big(\textstyle\frac{R}{s}\big)^{n-1}
	\left(\textstyle\frac{s^{(n-1)\gamma}\bar p(s)}{(R_sR^{n-1})^\gamma}\right)_{\!s}.
\eeq
%Alternatively, defining the weighted density and pressure profiles
%\[\hat\rho(x):=s^{n-1}\bar\rho(s)\qquad\qquad \hat p(s):=n^\gamma s^{(n-1)\gamma}\bar p(s),\]
%the radial EOM \eq{eom_rad} takes the more compact form
%\beq\label{eom_rad_2}
%	\ddot R=-\textstyle\frac{R^{n-1}}{\hat\rho(s)}
%	\left(\textstyle\frac{\hat p(s)}{[(R^n)_s]^\gamma}\right)_{\!s}.
%\eeq
It is routine to verify that any $C^2$ solution $R(t,s)$ of the radial EOM \eq{eom_rad} generates, via
\eq{rad_motn}, \eq{u1}, \eq{rho_formula}, and \eq{p_formula}, a radial $C^1$ Euler flow.
%Alternatively, one may  use the unknown 
%\[\mathcal R(t,s):= \textstyle\frac{R(t,s)}{s},\]
%in terms of which the motion is
%\[\bX(t,\by)=\mathcal R(t,s)\by,\]
%giving the EOM
%\beq\label{eom_rad_altn}
%	\ddot{\mathcal R}=-\textstyle\frac{\mathcal R^{n-1}}{s\bar\rho(s)}
%	\left(\textstyle\frac{\bar p(s)}{\big[(s\mathcal R)_s\mathcal R^{n-1}\big]^\gamma}\right)_{\!s}.
%\eeq

%%%%%%%%%%%%%%%%%%%%%%%%%%%%%%
%%%%%%%%%%%%%%%%%%%%%%%%%%%%%%
\section{Separated radial motions and affine radial motions}\label{septd_aff}
%%%%%%%%%%%%%%%%%%%%%%%%%%%%%%
%%%%%%%%%%%%%%%%%%%%%%%%%%%%%%
One case for which the Lagrangian EOM simplifies enough to be 
analyzed in detail is provided by {\em separated} radial motions, i.e., 
radial motions \eq{rad_motn} for which $R(t,s)$ separates:
\beq\label{sep_rad}
	r=R(t,s)=\alpha(t)f(s),
\eeq
where $\alpha(t)$ and $f(s)$ are scalar functions of the same sign (taken as positive in what follows).
The ansatz \eq{sep_rad} is a strong restriction on the motion.
It should not come as a surprise that it is not possible to freely
assign initial data for two independent thermodynamic variables in this setup,
cf.\ \eq{septd} and \eq{bar_p_bar_rho} below.

The use of a separation ansatz in this context goes back at least to McVittie \cite{mcvittie} and 
Keller \cite{kell} (see Remarks \ref{mcvittie_approach} and \ref{keller_approach} below). 
It necessarily leads to flows in which the velocity field 
at each time is linear in the spatial (Eulerian) coordinate:
With $u(t,r)\in\RR$ ($r=|\bx|$) denoting the radial velocity in physical space, 
\eq{sep_rad} gives
\[u(t,R(t,s))=\dot R(t,s)=\dot\alpha(t)f(s)=\textstyle\frac{\dot\alpha(t)}{\alpha(t)}R(t,s),\]
i.e.,
\beq\label{u_rad_aff}
	u(t,r)=\textstyle\frac{\dot\alpha(t)}{\alpha(t)}r.
\eeq
In what follows, separated motions \eq{sep_rad} with $f(s)=s$, i.e.,
\beq\label{rad_aff}
	R(t,s)=\alpha(t)s,
\eeq 
are referred to as {\em affine} radial motions. 
%%%%%%%%%%%%%%%%%%%%%%%%%%%%%%
\begin{remark}\label{mcvittie_approach}
        Radial flows with the velocity field linear in the (Eulerian) spatial coordinate were 
        considered by McVittie \cite{mcvittie}, Sedov \cite{sed}, and Keller \cite{kell}. 
        (For additional references to the Russian literature, see the endnotes 
        to Chapter IV in \cite{sed}.) 
        We note that such flows necessarily give separated 
        motions in the sense above: 
        If $u(t,r)=\beta(t)r$, then \eq{sep_rad} holds with 
        $\alpha(t)=\exp(\int_0^t\beta(\tau)\, d\tau)$ and $f(s)=R(0,s)$. 
\end{remark}
%%%%%%%%%%%%%%%%%%%%%%%%%%%%%%

Affine motions have been 
considered in a number of works, among them \cites{mcvittie,sid_2014,sid_2017,sid_2024,kell}; 
for further references and discussion, see \cites{sid_2024,sjh}.
The more general setting where $\alpha(t)$ is matrix-valued is considered by Sideris
\cite{sid_2017}. We shall restrict attention to radial motions so that $\alpha(t)\in\RR$
in what follows.

%%%%%%%%%%%%%%%%%%%%%%%%%%%%%%
\subsection{Separated vs.\ affine radial flows}\label{septd_vs_aff}
%%%%%%%%%%%%%%%%%%%%%%%%%%%%%%
We next argue that the class of Euler flows 
generated by separated radial motions \eq{sep_rad} coincides with the class generated by
(the seemingly more special) affine radial motions \eq{rad_aff}. This is reasonable 
given the freedom we have to choose reference profiles for the thermodynamic variables. 

To argue analytically for this,
we substitute \eq{sep_rad} into \eq{eom_rad} and separate variables to get that 
there is a separation constant $\lambda\in\RR$ such that (suppressing 
arguments)
\beq\label{septd}
	\ddot\alpha\alpha^{n(\gamma-1)+1}=\lambda=-\textstyle\frac{1}{\bar\rho f}
	\big(\textstyle\frac{f}{s}\big)^{n-1}
	\left(\textstyle\frac{s^{(n-1)\gamma}\bar p}{(f^{n-1}f')^\gamma}\right)_s,
\eeq
where we note that the right-hand equality defines $\bar \rho(s)$ in terms of  $\bar p(s)$
(for given $\lambda$ and $f(s)$). 

The special case with $\lambda=0$ is considered separately in Section \ref{aff_lam_0}.
Assuming for now that $\lambda\neq0$, we obtain the following construction scheme 
for generating radial Euler flows from separated solutions of the radial EOM \eq{eom_rad}:
\begin{enumerate}
	\item Fix $\lambda\neq0$ and solve $\ddot\alpha=\lambda\alpha^{n(1-\gamma)-1}$ with $\alpha(0)>0$.
	(This ODE is analyzed in the following sections.)
	\item Fix positive functions $\bar p(s)$ and $f(s)$, with $f(0)=0$. (The latter condition ensures that the 
	particle located at the center of motion remains there for all times).
	\item Define $\bar\rho$ according to the right-hand equality in \eq{septd}:
	\beq\label{ref_dens}
		\bar\rho(s):=-\textstyle\frac{1}{ \lambda f}
		\big(\textstyle\frac{f}{s}\big)^{n-1}
		\left(\textstyle\frac{s^{(n-1)\gamma}\bar p}{(f^{n-1}f')^\gamma}\right)_s.
	\eeq
	(Of course, to obtain a physically meaningful flow, $\lambda$, $f$, and $\bar p$ 
	must be such that the right-hand side of \eq{ref_dens} is non-negative.)
	With this, $R(t,s):=\alpha(t)f(s)$ solves the radial EOM \eq{eom_rad}.
	\item Letting $r\mapsto S(t,r)$ denote the inverse 
	map of $s\mapsto R(t,s)$, and defining
	\begin{align}
		\rho(t,r)&:=\textstyle\frac{\bar\rho(s)}{J(t,s)}\big|_{s=S(t,r)}\label{constrctd_rho}\\
		u(t,r)&=\textstyle\frac{\dot\alpha(t)}{\alpha(t)}r\label{constrctd_u}\\
		p(t,r)&:=\textstyle\frac{\bar p(s)}{J(t,s)^\gamma}\big|_{s=S(t,r)},\label{constrctd_p}
	\end{align}
	where
	\beq\label{J_sep}
		J(t,s)=\big(\textstyle\frac{R}{s}\big)^{n-1}R_s
		=\alpha(t)^n\big(\textstyle\frac{f(s)}{s}\big)^{n-1}f'(s),
	\eeq
	we obtain a radial  Euler flow for an ideal polytropic gas, 
	i.e., a solution to the system \eq{mass_rad}, \eq{mom_rad}, \eq{p_eqn_ideal_rad}.
\end{enumerate}
For the separated case \eq{sep_rad} under consideration, the standing assumption 
that the radial-position-to-radial-label map $r\mapsto S(t,r)$ is well-defined amounts 
to the assumption that the 
map $s\mapsto f(s)$ is invertible (with the required smoothness), so that
\[S(t,r)=f^{-1}\big(\textstyle\frac{\alpha(t)}{r}\big).\] 
With this assumption in force, we can define the functions $\hat p$ and $\tilde p$ by 
\beq\label{p_hat}
	\hat p(s):=\textstyle\frac{s^{(n-1)\gamma}\bar p(s)}{(f^{n-1}(s)f'(s))^\gamma},
\eeq
and
\beq\label{p_tilde}
	\tilde p(f(s)):=\hat p(s).
\eeq
Using \eq{J_sep} in \eq{constrctd_p} then gives
\beq\label{constrctd_p_2}
	p(t,r)=\textstyle\frac{1}{\alpha(t)^{n\gamma}}\hat p(s)\big|_{s=S(t,r)}
	=\textstyle\frac{1}{\alpha(t)^{n\gamma}}\tilde p(\textstyle\frac{r}{\alpha(t)}).
\eeq
Also, using \eq{p_hat} in \eq{ref_dens} gives
\[\bar\rho(s)= -\textstyle\frac{1}{\lambda f(s)}\big(\textstyle\frac{f(s)}{s}\big)^{n-1}\hat p'(s),\]
which, together with \eq{J_sep}, \eq{constrctd_rho}, and \eq{sep_rad}, gives
\beq\label{constrctd_rho_2}
	\rho(t,r)=-\textstyle\frac{\alpha(t)^{1-n}}{\lambda r}\frac{\hat p'(s)}{f'(s)}\big|_{s=S(t,r)}
	=-\textstyle\frac{\alpha(t)^{1-n}}{\lambda r}\tilde p'(\textstyle\frac{r}{\alpha(t)}).
\eeq
We conclude that any radial Euler flow generated by a  separated solution 
\eq{sep_rad} (with $f(s)$ $C^2$-invertible) of the radial EOM \eq{eom_rad}, can be expressed 
in the following form:
\begin{align}
		\rho(t,r)&:=-\textstyle\frac{\alpha(t)^{1-n}}{\lambda r}\tilde p'(\textstyle\frac{r}{\alpha(t)})\label{constrctd_rho_3}\\
		u(t,r)&=\textstyle\frac{\dot\alpha(t)}{\alpha(t)}r\label{constrctd_u_3}\\
		p(t,r)&:=\textstyle\frac{1}{\alpha(t)^{n\gamma}}\tilde p(\textstyle\frac{r}{\alpha(t)}).\label{constrctd_p_3}
\end{align}
On the other hand, if we start from the affine ansatz \eq{rad_aff}, viz.\ $R(t,s)=\alpha(t)s$, \eq{constrctd_rho}-\eq{J_sep}
give the radial Euler flow
\begin{align}
		\rho(t,r)&:=-\textstyle\frac{\alpha(t)^{1-n}}{\lambda r}\bar p'(\textstyle\frac{r}{\alpha(t)})\label{constrctd_rho_4}\\
		u(t,r)&=\textstyle\frac{\dot\alpha(t)}{\alpha(t)}r\label{constrctd_u_4}\\
		p(t,r)&:=\textstyle\frac{1}{\alpha(t)^{n\gamma}}\bar p(\textstyle\frac{r}{\alpha(t)}),\label{constrctd_p_4}
\end{align}
which is simply \eq{constrctd_rho_3}-\eq{constrctd_p_3} with $\tilde p=\bar p$.
As we are free to choose the reference pressure profile $\bar p(s)$, we conclude 
that the class of Euler flows generated by affine solutions of the radial EOM \eq{eom_rad} 
includes all Euler flows generated by separated solutions of the radial EOM \eq{eom_rad}. 

The analysis above assumed $\lambda\neq0$, but a similar (and simpler) 
analysis shows that the same conclusion holds when $\lambda=0$; we omit the details. We thus have:

%%%%%%%%%%%%%%%%%%%%%%%%%%%%%%
\begin{proposition}\label{septd_vs_affine}
	Any radial Euler flow generated by a separated solution $R(t,s)=\alpha(t)f(s)$ 
	(with $f$ $C^2$-invertible) of the radial EOM
	\eq{eom_rad} can be realized from a radial affine solution $R(t,s)=\alpha(t)s$ by suitably
	choosing the reference pressure profiles $\bar p(s)$ and $\bar\rho(s)$.
\end{proposition}
%%%%%%%%%%%%%%%%%%%%%%%%%%%%%%

%%%%%%%%%%%%%%%%%%%%%%%%%%%%%%
\begin{remark}\label{keller_approach}
	Keller \cite{kell} applies a slightly different setup: He considers radial motions of an ideal gas
	and opts to formulate the {\em EOM} in terms of the Lagrangian mass coordinate $h$, i.e.,
	\[h=\int_0^{r(t,h)}\rho(t,\xi)\xi^{n-1}\,d\xi,\] 
	so that $r(t,h)$ is proportional to the radius which bounds 
	the mass $h$ at time $t$; cf.\ Section 18 in \cite{cf}.
	Keller then makes the separation ansatz $r(t,h)=j(t)f(h)$ (the $f$ here is 
	different from our $f$ above).
	A somewhat involved analysis of the resulting ODE for the spatial part $f(h)$ 
	again leads to Euler flows of the form \eq{constrctd_rho_4}-\eq{constrctd_p_4}
	(with $j(t)\equiv\alpha(t)$, cf.\ Eqns.\ (10) and (24)-(26) in \cite{kell}). 
	
	Starting instead with the {\em EOM} \eq{eom_lagr_rho_p} and 
	imposing the affine ansatz \eq{rad_aff} simplifies the argument and 
	generates the same class of Euler flows. In particular, with the assumption 
	\eq{rad_aff} there is no equation to be solved for the spatial part
	of the motion. This illustrates the gauge freedom one has in the Lagrangian formulation.
	
	McVittie \cite{mcvittie}  employed an alternative (Eulerian) approach via 
        ``degenerate Einsteinian gravitational potentials'' combined with a separation ansatz, 
        again obtaining radial flows with linear velocity profiles. 
        See \cite{sjh} for a concise summary and extensions of this approach.
\end{remark}
%%%%%%%%%%%%%%%%%%%%%%%%%%%%%%

%%%%%%%%%%%%%%%%%%%%%%%%%%%%%%
%%%%%%%%%%%%%%%%%%%%%%%%%%%%%%
\section{Affine radial motions, cumulative behavior}\label{affine_anal}
%%%%%%%%%%%%%%%%%%%%%%%%%%%%%%
%%%%%%%%%%%%%%%%%%%%%%%%%%%%%%
From now on we focus on radial affine solutions $R(t,s)=\alpha(t)s$ in \eq{rad_aff} 
with the goal of analyzing their behavior and in particular the possibility of accumulation. 
To do so we first fix 
the radial Lagrangian variable $s$ to be the initial (radial) position of a particle,  i.e.,
\beq\label{init_alpha}
	\alpha(0)=1\qquad\Leftrightarrow\qquad R(0,s)=s.
\eeq
According to the separated form \eq{septd} of the EOM we then 
obtain the equations
\beq\label{alpha_eqn}
	\ddot\alpha=\lambda\alpha(t)^{-n(\gamma-1)-1}
\eeq
and
\beq\label{bar_p_bar_rho}
	\bar p'(s)=-\lambda s\bar\rho(s),
\eeq
where $\lambda\in\RR$ is the separation constant.
We note that the initial pressure and density profiles $\bar p$, $\bar\rho$ 
cannot be freely assigned if $\lambda\neq0$, while $\bar p(s)$ is constant 
and $\bar\rho(s)$ is freely assignable (subject to positivity) if $\lambda=0$. 

Before analyzing radial affine motions with vanishing and non-vanishing 
separation constant $\lambda$ we make some preliminary observations.

First, as detailed in Section \ref{aff_lam_0}, if both $\lambda$ and 
$\dot\alpha(0)$ vanish, then we obtain a trivial (i.e., stationary and quiescent) 
flow where each particle remains in its initial position. In all that follows, 
trivial flows are excluded from consideration.

With \eq{alpha_eqn} solved for a given assignment of $\dot\alpha(0)$,
$R(t,s)=\alpha(t)s$ gives the particle path (radial position) of a particle starting 
from $r=s$ at $t=0$. Particle $s$ then moves with radial velocity
$\dot\alpha(t)s$ and radial acceleration $\ddot\alpha(t)s$. In particular, 
the initial radial velocity and acceleration of particle $s$ are
$\dot\alpha(0)s$ and (according to \eq{init_alpha}${}_1$ and \eq{alpha_eqn}) 
$\ddot\alpha(0)s=\lambda s$, respectively. For later reference we record:
%%%%%%%%%%%%%%%%%%%%%%%%%%%%%%%%%%%
\begin{observation}\label{obs1}
	Consider a (non-trivial) radial affine motion $R(t,s)=\alpha(t)s$ with separation constant 
	$\lambda$ and which is defined on all of space (i.e., for all $s\geq0$). 
	\begin{itemize}
	\item[(i)] If $\dot\alpha(0)\neq0$ then it generates an Euler flow with unbounded 
	initial velocity in the far-field $s\to\infty$. 
	\item[(ii)] If $\lambda\neq0$ then it generates an Euler flow with unbounded 
	initial acceleration in the far-field $s\to\infty$. 
\end{itemize}
\end{observation}
%%%%%%%%%%%%%%%%%%%%%%%%%%%%%%%%%%%
As is evident from the particular separated form $R(t,s)=\alpha(t)s$, one type 
of singularity in affine flows occurs whenever the function $\alpha(t)$ 
vanishes at some time $t_c>0$. This means that all fluid 
particles accumulate at the center of motion $r=0$ at time $t=t_c$, a phenomenon we 
refer to as {\em accumulation} or {\em cumulative} behavior:
%%%%%%%%%%%%%%%%%%%%%%%%%%%%%%%%%%%%
\begin{definition}\label{def_accmln}
	A solution of the Euler system \eq{mass_cons}-\eq{energy_cons} is said to suffer 
	{\em accumulation} if all fluid particles arrive at a common point at a common time.
\end{definition}
%%%%%%%%%%%%%%%%%%%%%%%%%%%%%%%%%%%%
Accumulation implies the generation of a Dirac delta in the density field, and 
it is thus a stronger singularity than the amplitude blowup which occurs in 
imploding self-similar flows where no Dirac delta is generated at time of blowup.
E.g., in Guderley's original imploding shock solution \cites{gud,jls}, the density field takes a 
globally constant value at time of collapse. In Hunter's collapsing cavity flow \cite{hun_60},
as well as in the more recently constructed continuously imploding flows \cites{jt2,jt3,mrrs1}, 
the density field suffers blowup at time of collapse. However, the density field remains an $L^1_{loc}$-function
at all times and no Dirac delta appears. The same applies to the other primary flow 
variables velocity, pressure, temperature, and sound speed.

As observed by Sedov \cite{sed} and Keller \cite{kell}, certain radial affine motions do generate 
flows suffering accumulation. Unsurprisingly however, and as we shall see below,
such solutions necessarily suffer from certain unphysical features.
In Section \ref{disc} we comment on the possibility of constructing 
physically acceptable examples of accumulation by modifying the data
in affine cumulative flows.

We note that for radial affine flows, any accumulation is necessary ``total'' in the sense that 
{\em all} fluid particles arrive at the origin at a common time $t_c$ (where $\alpha(t_c)=0$). 
One can imagine other,
milder types of cumulative behavior. E.g., only a certain part of the total mass could 
accumulate at one point at a single time, or the accumulation could be continuous 
in time with particles reaching a common point at distinct times. We are not aware 
of examples of such milder forms of accumulation for the Euler system in the absence of 
forces other than the internal thermodynamical pressure. (In the work \cite{ghj} on collapsing 
self-gravitating gaseous stars, the authors refer to continuous accumulation as
fragmented, continued, or inhomogeneous collapse.)
Another type of singular behavior is related to Observation \ref{obs1} and concerns
far-field behavior ($r\to\infty$).

Finally, under certain conditions, affine solutions provide examples of
bounded and dynamically changing fluid spheres (or ellipsoids \cite{sid_2017}) 
surrounded by vacuum. This scenario 
has been analyzed in detail in recent works \cites{sid_2014,sid_2017}; it is commented on briefly
in Section \ref{lambda_pos}.

In the next two sections we consider radial affine motions with 
vanishing and non-vanishing separation constant $\lambda$, respectively.
We determine the general functional form of the flow variables $\rho,u,p$,
and analyze particle trajectories and characteristics in the resulting flows.
Particular attention is paid to identifying the cases where accumulation 
occurs and the effects of the pressure gradient and far-field behavior. 
We shall see that accumulation can occurs in different manners (entirely due to 
unbounded initial velocities or due to an adverse pressure gradient).
Concrete examples illustrating the findings are given.

%%%%%%%%%%%%%%%%%%%%%%%%%%%%%%%%%%%
%%%%%%%%%%%%%%%%%%%%%%%%%%%%%%%%%%%
\section{Radial affine flows with separation constant $\lambda=0$}\label{aff_lam_0}
%%%%%%%%%%%%%%%%%%%%%%%%%%%%%%%%%%%
%%%%%%%%%%%%%%%%%%%%%%%%%%%%%%%%%%%
With $\lambda=0$ \eq{init_alpha}-\eq{alpha_eqn} give 
\[\alpha(t)=1+bt,\qquad\text{where $b=\dot\alpha(0)$,}\]
and \eq{bar_p_bar_rho} gives $\bar p\equiv p_0$ (constant). We restrict attention to the 
case where $p_0$ is a strictly positive constant. 
Also, the case $b=0$ corresponds to a stationary and quiescent fluid; it 
is not considered further in what follows.

According to \eq{rho_formula}, \eq{u_rad_aff}, \eq{alpha_eqn}, and \eq{p_formula},
we thus obtain the following family of radial affine Euler flows parametrized by the initial 
density profile $\bar\rho(s)$:
\begin{align}
	\rho(t,r)&=\textstyle\frac{1}{(1+bt)^n}\bar\rho(\frac{r}{1+bt})\label{rad_aff_density_0}\\
	u(t,r)&=\textstyle\frac{br}{1+bt}\qquad\qquad\qquad\qquad\text{($b\neq0$, $p_0>0$)}\label{rad_aff_speed_0}\\
	p(t,r)&=\textstyle\frac{p_0}{(1+bt)^{n\gamma}}.\label{rad_aff_pressure_0}
\end{align}
%It is readily verified that \eq{rad_aff_density_0}-\eq{rad_aff_pressure_0} (with \eq{alpha_eqn} in force),
%satisfy the radial Euler equations \eq{mass_rad}, \eq{mom_rad}, and \eq{p_eqn_ideal_rad}.
Since the pressure is a strictly positive constant at each fixed time, 
this class of solutions describes non-vacuum flows where fluid is present 
in all of space at all times.
% (i.e., balls of fluid surrounded by an exterior vacuum is not an option in this case). 
A minimal restriction on the initial density profile is positivity: $\bar\rho(s)\geq0$ for all $s\geq 0$.

The particle starting from radial position $s$ follows the straight line path 
\[R(t,s)=s(1+bt),\] 
and we obtain
two distinct behaviors depending on the sign of $b$: 
$b>0$ gives an expanding and rarefying flow defined for all times $t\geq0$, while $b<0$ yields a converging 
and accumulating flow in which all fluid particles reach the origin at the critical time $t_c=-\frac{1}{b}$.
The former case is not considered further, and for the latter case we do not attempt to propagate the 
solution beyond $t=t_c$. We note the following features:
\begin{enumerate}
	\item Whenever $b\equiv\dot\alpha(0)\neq0$ the initial data suffers from the unphysical 
	property of unbounded velocities of particles in the far-field (cf.\ Observation \ref{obs1}). 
	\item The cumulative behavior that occurs for $b<0$ is hardly surprising: 
	Mass is sent toward the origin at initial time with unbounded speed  in the far-field, and in 
	such a manner that the pressure remains constant in space at all times $t<t_c$. 
	No pressure gradient develops to push fluid either away from or 
	toward the origin. In particular, accumulation in this case is not ``driven'' by a pressure gradient; 
	rather it is an inertial effect and a consequence of the unbounded initial velocity field.
	(Below we shall see that accumulation can occur in cases 
	where the fluid is initially expanding and where a positive pressure gradient is
	responsible for the cumulative behavior; cf. Case (B2a).) 
	\item If we require the total mass to be finite, then the initial density profile $\bar\rho(s)$ must
	 decay to zero as $s\to\infty$. However, as the initial pressure profile is constant, 
	 this implies that the initial temperature field \ $\bar\theta(s)\propto\frac{p_0}{\bar\rho(s)}$ 
	 is unbounded.
	 \item If instead we assume that $\bar\rho(s)$ remains bounded away from vacuum, 
	 then the total mass is infinite.
%	 \footnote{Infinite total mass also occurs in certain radially 
%	 converging self-similar solutions that display amplitude blowup (but not accumulation). 
%	 Density, pressure, temperature, and velocity all decay to zero in the far-field for these 
%	 self-similar solutions.} 
\end{enumerate}
%In our view, property (1) renders solutions of the form \eq{rad_aff_density_0}-\eq{rad_aff_pressure_0} 
%(with $b\neq0$) unphysical.
We next consider the characteristics in the flow. This will be used in Section \ref{disc}
to discuss the possibility of generating physically acceptable examples of accumulation. 
We focus on 1-characteristics $r=r(t)$, satisfying $\dot r=u-c$, in the cumulative case $b<0$.
(Here $c\geq0$ is the local sound speed, cf.\ \eq{sound_speed}.)

Before proceeding with the details we note that, by definition, 1-characteristics propagate slower (i.e., they 
propagate faster
toward the center of motion) than do particle trajectories (which satisfy $\dot r=u$). It is 
therefore reasonable to expect that 1-characteristics, in the cumulative case under consideration,
will reach the center of motion $r=0$ strictly before accumulation occurs at time $t_c=-\frac{1}{b}$.
As the following calculations and examples show, this may or may not be the case:
In certain cases, not only do all particle trajectories, but also all 1-characteristics 
accumulate at $r=0$ at the critical time $t=t_c$; see Example \ref{ex_lam_0} and Figures \ref{paths_1}-\ref{paths_2}. (The same holds 
in certain cases with $\lambda\neq0$; see Example \ref{ex_lam_neg}.)

According to \eq{rad_aff_density_0}-\eq{rad_aff_pressure_0}, a 1-characteristic
$r=r(t)$ satisfies
\beq\label{1_chars}
	\dot r=u(t,r)-c(t,r)=\textstyle\frac{br}{1+bt}
	-\sqrt{\gamma p_0}\cdot\textstyle
	\frac{(1+bt)^{-\frac{n(\gamma-1)}{2}}}{\sqrt{\bar\rho(\frac{r}{1+bt})}}.
\eeq
Introducing the auxiliary variable
\[F(t):=\textstyle\frac{r(t)}{1+bt},\] 
we obtain the ODE
\[\sqrt{\bar\rho(F)}\dot F=-\sqrt{\gamma p_0}\cdot(1+bt)^{-\frac{n(\gamma-1)}{2}-1},\]
or, integrating from initial time to time $t$ (using $F(0)=r(0)$),
\beq\label{1-char_lam_0}
	\int_{F(t)}^{r(0)}\sqrt{\bar\rho(\xi)}\,d\xi=m
	\left[(1+bt)^{-\frac{n(\gamma-1)}{2}}-1\right]
	\qquad \qquad \big(m=\textstyle\frac{2\sqrt{\gamma p_0}}{|b|n(\gamma-1)}\big).
\eeq
As the following example shows, different choices for the initial density 
profile $\bar\rho$ can yield qualitatively different types of behaviors of 
1-characteristics near the center of motion as the critical time $t_c=-\frac{1}{b}$
is approached.
%%%%%%%%%%%%%%%%%%%%%%%%%%%%%%%%%%%
\begin{example}\label{ex_lam_0}
	Consider the initial density profile
	\[\bar\rho(s)=s^k,\qquad k\in\RR,\, k\neq-2,\]
	which yields the 1-characteristics
	\[r(t)=r(0)(1+bt)\cdot
	\left\{1+r(0)^{-\frac{k+2}{2}}\textstyle\frac{k+2}{2}m
	\left(1-(1+bt)^{-\frac{n(\gamma-1)}{2}}\right)\right\}^\frac{2}{k+2}.\]
	With $k=0$, i.e., constant initial density, we obtain 1-characteristics of the form 
	\beq\label{1_char_const_init_dens}
		r(t)=m(1+bt)\cdot\left\{\left(1+\textstyle\frac{r(0)}{m}\right)-(1+bt)^{-\frac{n(\gamma-1)}{2}}\right\}.
	\eeq
	As $t$ increases from initial time $t=0$ the expression within the curly brackets tends to zero
	at some time strictly prior to $t_c=-\frac 1 b$. 
	Hence, each 1-characteristic reaches the center of motion $r=0$ 
	strictly before the critical accumulation time $t_c$; in particular, accumulation of 1-characteristics 
	does not occur. (It is true, though, that the time of arrival at $r=0$ of the 1-characteristic 
	starting from $r(0)$, increases toward $t_c$ as $r(0)\to+\infty$.)

	In the particular case when $n(\gamma-1)=2$ (e.g.,
	$n=3$, $\gamma=\frac 5 3$) the 1-characteristics \eq{1_char_const_init_dens} are straight lines
	\[r(t)=m\cdot\left\{\left(1+\textstyle\frac{r(0)}{m}\right)(1+bt)-1\right\}
	\qquad\qquad (m=\textstyle\frac{\sqrt{\gamma p_0}}{|b|});\]
	see Figure \ref{paths_1}.
	%%%%%%%%%%%%%%%%%%%
	%	FIGURE 
	%%%%%%%%%%%%%%%%%%%
	\begin{figure}
		\centering
		\includegraphics[width=9cm,height=7cm]{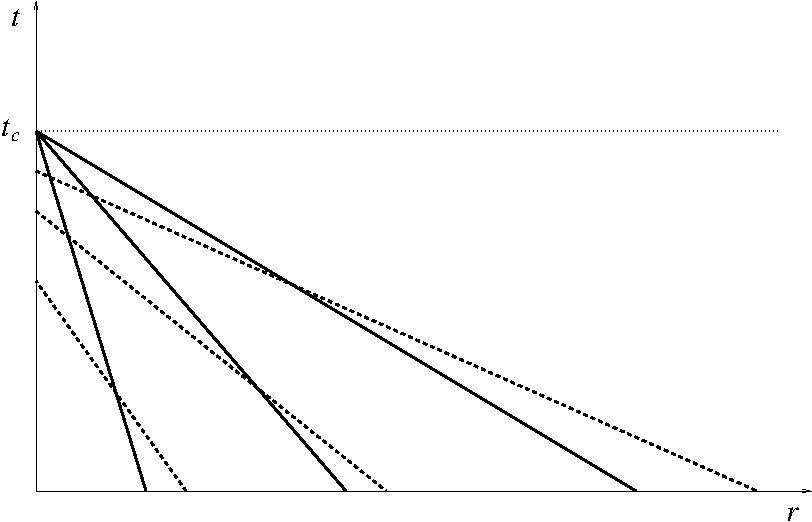}
		\caption{Schematic picture of particle trajectories (solid) and 1-characteristics (dashed) in 
		radial affine flow with separation constant $\lambda=0$. The initial density field  
		is constant and the relation $n(\gamma-1)=2$ is assumed (so that also the 
		1-characteristics are straight lines); cf. Example \ref{ex_lam_0}.}\label{paths_1}
	\end{figure} 
%%%%%%%%%%%%%%%%%%%%%%%%%%%%%%%%%%%%%%%%%%%%

	A different type of behavior results from the unbounded initial density profile 
	\[\bar\rho(s)=s^{-2}.\]
	We consider the 3-dimensional case $n=3$ so that $\bar\rho\in L^1_{loc}(\RR^+,s^2ds)$,
	i.e., initial density is locally integrable in physical space in this case.
	From \eq{1-char_lam_0} we get that the 1-characteristics $r(t)$ in this case are given by 
	\[r(t)=r(0)(1+bt)\exp\left[1-(1+bt)^{-\frac{3(\gamma-1)}{2}}\right],\]
	see Figure \ref{paths_2}. This shows that all 1-characteristics approach $r=0$ with vanishing 
	speed as $t\uparrow t_c=-\frac 1 b$. Thus, in this case, all 1-characteristics as well as
	all particle trajectories accumulate at the origin at the same time. (A similar situation 
	occurs in certain affine flows with separation constant $\lambda<0$; see Example 
	\ref{ex_lam_neg}).
	%%%%%%%%%%%%%%%%%%%
	%	FIGURE 
	%%%%%%%%%%%%%%%%%%%
	\begin{figure}
		\centering
		\includegraphics[width=9cm,height=7cm]{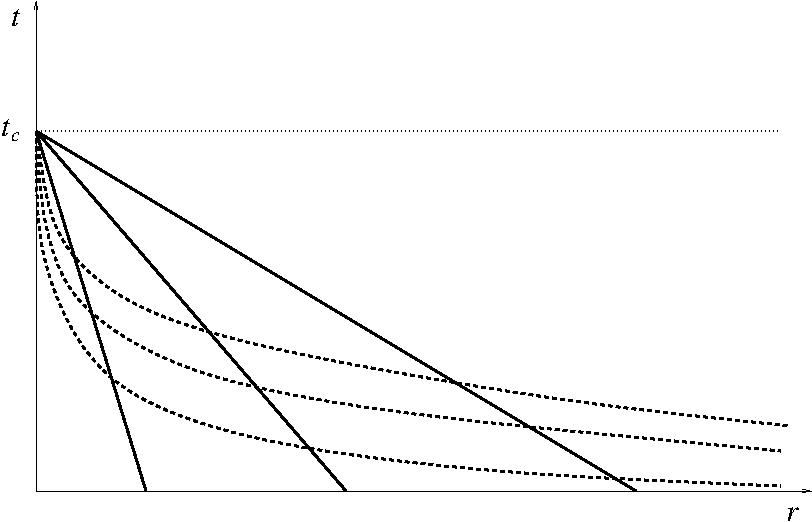}
		\caption{Schematic picture of particle trajectories (solid) and 1-characteristics (dashed) in 
		radial affine flow with separation constant $\lambda=0$. The initial density field  
		is singular ($\bar\rho(s)=s^{-2}$) and $n=3$. In this case both particle trajectories 
		and 1-characteristics accumulate at $r=0$ at time $t_c$; cf. Example \ref{ex_lam_0}.}\label{paths_2}
	\end{figure} 
\end{example}
%%%%%%%%%%%%%%%%%%%%%%%%%%%%%%%%%%%

Summing up, the radial affine solutions corresponding to separation constant $\lambda=0$
given in \eq{rad_aff_density_0}-\eq{rad_aff_pressure_0} suffer accumulation whenever $b<0$,
while the corresponding 1-characteristics may or may not accumulate. The pressure is globally 
constant at each time, and accumulation is entirely due to the fact that the initial velocity 
field is unbounded.

%%%%%%%%%%%%%%%%%%%%%%%%%%%%%%%%%%%
%%%%%%%%%%%%%%%%%%%%%%%%%%%%%%%%%%%
\section{Radial affine flows with separation constant $\lambda\neq 0$}\label{aff_lam_not_0}
%%%%%%%%%%%%%%%%%%%%%%%%%%%%%%%%%%%
%%%%%%%%%%%%%%%%%%%%%%%%%%%%%%%%%%%
For radial affine flows with separation constant $\lambda\neq0$ and reference (initial)
pressure profile $\bar p$, the flow 
variables in the Eulerian frame are given by 
\eq{constrctd_rho_4}-\eq{constrctd_p_4}:
\begin{align}
	\rho(t,r)&=-\textstyle\frac{\alpha(t)^{1-n}}{\lambda r}
	\bar p\,'(\textstyle\frac{r}{\alpha(t)})\label{rad_aff_density}\\
	u(t,r)&=\textstyle\frac{\dot\alpha(t)}{\alpha(t)}r
	\qquad\qquad\qquad\qquad\qquad\qquad\text{($\lambda\neq0$)}\label{rad_aff_speed}\\
	p(t,r)&=\alpha(t)^{-n\gamma}\bar p(\textstyle\frac{r}{\alpha(t)}),\label{rad_aff_pressure}
\end{align}
where $\alpha(t)$ solves \eq{alpha_eqn}. Note that only non-negative functions 
$\alpha(t)$ are relevant (since both $s$ and $R(t,s)=\alpha(t)s$ are non-negative 
radial positions).

In contrast to the case with vanishing separation constant (Section \ref{aff_lam_0}), 
the set of solutions is now parametrized by the initial pressure profile $\bar p(s)$. 
It follows from \eq{rad_aff_density} that we must choose $\lambda$ and $\bar p(s)$ so that
\beq\label{p_rho_reln}
	\sgn(\bar p\,')=-\sgn(\lambda).
\eeq
In particular, pressure gradients are necessarily present when $\lambda\neq0$.
It follows from \eq{rad_aff_density} and \eq{rad_aff_pressure} (with $\alpha(t)$ positive) 
that the density field then remains positive and the pressure gradient never changes sign.

For later reference we record the ODE for 1- and 3-characteristics 
$r=r(t)$ in flows described by \eq{rad_aff_density}-\eq{rad_aff_pressure}. 
We have
\[\dot r=u\pm c=u\pm\sqrt{\textstyle\frac{\gamma p}{\rho}}
	=\dot\alpha\textstyle\frac{r}{\alpha}\pm\sqrt{\gamma|\lambda|}
	\alpha^{-\frac{n(\gamma-1)}{2}}
	\sqrt{\textstyle\frac{r\bar p(\frac{r}{\alpha})}{\alpha\bar |p'(\frac{r}{\alpha})|}},\]
or, in terms of $F(t)=\frac{r(t)}{\alpha(t)}$,
\beq\label{1_3_chars_f}
	\dot F\sqrt{\textstyle\frac{|\bar p'(F)|}{f\bar p(F)}}
	=\pm\sqrt{\gamma|\lambda|}\alpha^{-\frac{n(\gamma-1)}{2}-1}.
\eeq
%Also, for an ideal polytropic gas we have that the temperature field $\theta$ satisfies:
%\[\theta(t,r)\propto\textstyle\frac{p}{\rho}
%=-\frac{\lambda}{\alpha(t)^{n(\gamma-1)}}\frac{s\bar p(s)}{\bar p'(s)}\big|_{s=\frac{r}{\alpha(t)}}.\]
%In particular, the initial temperature field satisfies
%\beq\label{init_temp}
%	\theta(0,r)=\bar\theta(r)\propto -\lambda\frac{r\bar p(r)}{\bar p'(r)}.
%\eeq
%It is reasonable to require that the latter is a bounded function. 
Next, with $\lambda\neq0$ and $\alpha(t)$ positive, it is convenient to consider 
the ODE \eq{alpha_eqn} for $\alpha(t)$, viz.\ $\ddot\alpha=\lambda\alpha(t)^{-n(\gamma-1)-1}$,
separately for $\lambda\gtrless0$:

\medskip

\begin{itemize}
	\item[(A)] $\lambda>0\quad\Rightarrow\quad \dot\alpha(t)$ increases in time;
	\item[]
	\item [(B)] $\lambda<0\quad\Rightarrow\quad \dot\alpha(t)$ decreases in time.
\end{itemize}

\medskip

\noindent
Multiplying \eq{alpha_eqn} by $\dot \alpha$ and integrating once gives
\beq\label{once_intgrtd}
	(\dot\alpha)^2=a-\textstyle\frac{2\lambda}{n(\gamma-1)}\alpha^{-n(\gamma-1)},
\eeq
where $a$ is a constant of integration. Having chosen
the Lagrangian coordinate $s$ to denote initial radial position (so that $\alpha(0)=1$),
\eq{once_intgrtd} shows that we must require
\beq\label{a_constr}
	a\geq\textstyle\frac{2\lambda}{n(\gamma-1)},
\eeq
which is assumed in all that follows. The particular $a$-value
\beq\label{a_0}
	a_0:=\textstyle\frac{2\lambda}{n(\gamma-1)}
\eeq
gives $\dot\alpha(0)=0$, i.e., the fluid is initially at rest.

Depending on the signs of $a$ and $\lambda$ we obtain 
different types of flow behavior. Note that we can obtain $\alpha(t)$ implicitly
 by integrating \eq{once_intgrtd}; in certain cases $\alpha(t)$ can be determined 
 explicitly (e.g., when $n(\gamma-1)=1,\, 2$).

%%%%%%%%%%%%%%%%%%%%%%%%%%%%%%
\subsection{Case (A): $\lambda>0$}\label{lambda_pos}
%%%%%%%%%%%%%%%%%%%%%%%%%%%%%%
We first observe that \eq{once_intgrtd}, with $\lambda>0$, implies that 
no solution exhibits accumulation in this case: $\alpha(t)$ is 
uniformly bounded away from zero along every solution of \eq{once_intgrtd}
(see Figure \ref{Figure_1}). 
Furthermore, for each choice of integration 
constant $a\geq a_0$, the solution curve in the $(\alpha,\dot\alpha)$-plane 
has the horizontal asymptotes $\dot\alpha=\pm\sqrt{a}$ as $\alpha\to\infty$.

Having chosen an $a$ satisfying \eq{a_constr} (in particular, $a>0$), 
let $\alpha_*=\alpha_*(a)$ be the $\alpha$-value for which the 
right-hand side of \eq{once_intgrtd} vanishes, i.e.,
\beq\label{alpha_star}
	\alpha_*=(\textstyle\frac{2|\lambda|}{n(\gamma-1)|a|})^\frac{1}{n(\gamma-1)}.
\eeq
(The absolute values on $\lambda$ and $a$ are included so that the same formula 
applies in Case (B).)
Note that \eq{a_constr} gives $\alpha_*\leq1$, and $\alpha_*=1$ corresponds to 
$a=a_0$.
We then solve \eq{once_intgrtd} for $\dot\alpha$ and integrate from time $t=0$ and use $\alpha(0)=1$ to get
\beq\label{alpha_formula}
	\int_1^{\alpha(t)} \textstyle\frac{\xi^{\frac{n(\gamma-1)}{2}}}
	{\sqrt{\xi^{n(\gamma-1)}-\alpha_*^{n(\gamma-1)}}}\,d\xi
	=\pm\sqrt{a}t
\eeq
%%%%%%%%%%%%%%%%%%%%%%%%%%%%%%
\begin{remark}\label{intgrbl}
	The integrand on the left-hand side of \eq{alpha_formula} is integrable at $\xi=\alpha_*+$.
\end{remark}
%%%%%%%%%%%%%%%%%%%%%%%%%%%%%%
%%%%%%%%%%%%%%%%%%%
%	FIGURE 
%%%%%%%%%%%%%%%%%%%
\begin{figure}
	\centering
	\includegraphics[width=9cm,height=9cm]{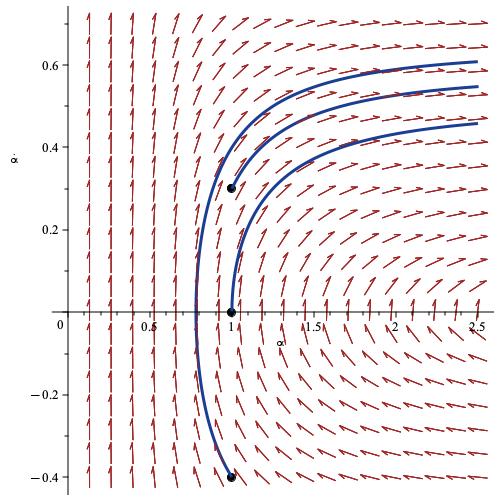}
	\caption{Case (A): Three $(\alpha,\dot\alpha)$-trajectories for \eq{once_intgrtd} with 
	$\lambda>0$. Solid dots indicate initial values (with $\alpha(0)=1$)
	and arrows indicate direction of flow as
	time increases. The parameter values are $n=3$,
	$\gamma=\frac{5}{3}$, $\lambda=\frac{1}{4}$.}\label{Figure_1}
\end{figure} 
%%%%%%%%%%%%%%%%%%%%%%%%%%%%%%%%%%%%%%%%%%%%

Figure \ref{Figure_1} displays three instances of Case (A) for different 
choices of the initial value $\dot\alpha(0)$ (positive, vanishing, and negative, respectively). 
With $\dot\alpha(0)\geq0$ the `$+$' 
sign should be used in \eq{alpha_formula} and the 
whole $(\alpha,\dot\alpha)$-trajectory is located in the 
upper half-plane. In this case $\alpha(t)>\alpha(0)=1$ 
and $\dot\alpha(t)>0$ for all $t>0$. We thus
obtain a radial flow in which all particles accelerate away from the center of motion for 
all times. According to \eq{once_intgrtd}, the particle initially at radius $s$ propagates 
with a speed which approaches the asymptotic value $\sqrt{a}s$.

If $\dot\alpha(0)<0$, then the $(\alpha,\dot\alpha)$-trajectory starts out in the 
lower half of the $(\alpha,\dot\alpha)$-plane, and the flow is a slightly more involved.
Let $t_*=t_*(a)$ be such that $\alpha(t_*)=\alpha_*$ (so that $\dot\alpha(t_*)=0$), 
and note that $t_*$ is finite by Remark \ref{intgrbl}. Then the `$-$' sign in \eq{alpha_formula}
should be used for $t<t_*$, and the `$+$' sign should be used for $t>t_*$. 
We then obtain a flow in which all fluid particles move toward the origin and 
decelerate for $t<t_*$; global stagnation (vanishing velocity) occurs at time 
$t=t_*$, and then all particles accelerate away from the origin for $t>t_*$,
particle $s$ propagating with asymptotic speed $\sqrt{a}s$.

%These conclusions are all contained in Keller's analysis \cites{kell}. 

Next, consider  the thermodynamic fields.
According to \eq{p_rho_reln} (with $\lambda>0$), the initial pressure profile 
$\bar p(s)$ must be a decreasing function, and we can arrange two different situations.
\begin{enumerate}
	\item[(Aa)] If $\bar p(s)$ is defined with $\bar p'(s)<0$ for all $s>0$, then 
	$\bar \rho(s)$ is defined and strictly positive for all $s>0$, i.e., 
	the gas fills all of space. However, \eq{rad_aff_speed} then
	yields unbounded speed of particles in the far-field $r\to\infty$ 
	at all times (except at time of stagnation $t_*$ if this occurs).
	\item[(Ab)] The other possibility is to set up a situation where the gas is initially 
	located within a ball of finite radius $s_0$, and with a vacuum on the outside.
	For this the flow needs to have vanishing pressure along a particle trajectory.
	We therefore prescribe a positive and decreasing pressure profile $\bar p(s)$ on $s\in[0,s_0]$ with 
	$\bar p(s_0)=0$. Next, $\bar p(s)$ determines the initial 
	density profile according to \eq{bar_p_bar_rho}, in particular at $s=s_0$. 
	Different values of $\bar p'(s_0)$ generate different physical properties (e.g., 
	of the entropy field) along
	the vacuum interface $r=\alpha(t)s_0$ (depending on whether $\bar p'(s_0)\leq0$ 
	vanishes or not); we refer to \cites{jt4,sid_2014,rickard_2021,rhj_2021} for further details. 
	Regardless of this, the vacuum interface propagates with acceleration $\ddot\alpha(t)s_0$,
	which is finite and non-vanishing according to \eq{alpha_eqn}. The
	so-called physical boundary condition along the free interface 
	is therefore satisfied \cites{jt4,liu_96,sid_2014,rickard_2021,rhj_2021}. 
\end{enumerate}
We re-iterate the unphysical feature of an unbounded velocity field for any affine 
flow defined in all of space in Case (A) (i.e., (Aa) above).
As the following example shows, this behavior occurs also in cases with 
seemingly reasonable initial data.
%%%%%%%%%%%%%%%%%%%%%%%%%%%%%%%%%%%%%%%%%%%%
\begin{example}\label{instant_blowup}
	With $\lambda>0$, choose $a=a_0$ (cf.\ \eq{a_0}) so that the fluid is initially at rest, 
	and let the initial pressure profile be a Gaussian, $\bar p(s):=p_0e^{-s^2}$ ($p_0>0$ constant).  
	According to \eq{bar_p_bar_rho}, the initial density profile is also a Gaussian, 
	$\bar\rho(s)=\frac{2p_0}{\lambda}e^{-s^2}$, while the initial temperature field is constant,  
	$\bar\theta(s)\equiv\frac{\lambda}{2c_v(\gamma-1)}>0$, (cf.\ \eq{pressure1}-\eq{polytr}).
	However, these pressure and density fields generate an initially unbounded acceleration 
	of the fluid in the far-field (cf.\ Observation \ref{obs1} (ii)), and this 
	results in a velocity field which is unbounded in the far-field at all times $t>0$.
\end{example}
%%%%%%%%%%%%%%%%%%%%%%%%%%%%%%%%%%%%%%%%%%%%

%%%%%%%%%%%%%%%%%%%%%%%%%%%%%%%%%%%%%%%%%%%%
\begin{remark}\label{instant_blowup_rmk}
	The ``instantaneous blowup'' of the velocity field in Example \ref{instant_blowup}
	occurs in particular for 1-dimensional flows. By choosing the constant $p_0$
	small (for a given $\lambda>0$), we can arrange that the conserved quantities 
	(density, linear momentum, total energy) are initially arbitrarily small  
	in $L^1(\RR)$ and $BV(\RR)$. Specifically, for the solution in Example 
	\ref{instant_blowup} with $n=1$, we have from \eq{rad_aff_density}-\eq{rad_aff_pressure} that
	\begin{align*}
		\rho(t,r)&=\textstyle\frac{2p_0}{\lambda \alpha(t)}\exp\big(-\frac{r^2}{\alpha(t)^2}\big)\\
		(\rho u)(t,r)&=\textstyle\frac{2p_0r\dot\alpha(t)}{\lambda\alpha(t)^2}\exp\big(-\frac{r^2}{\alpha(t)^2}\big)\\
		\big(\rho(\eps+\textstyle\frac{u^2}{2})\big)(t,r)
		&=p_0\big[\textstyle\frac{\alpha(t)^{-\gamma}}{\gamma-1}
		+\frac{(\dot\alpha(t))^2r^2}{\lambda\alpha(t)^3}\big]\exp\big(-\frac{r^2}{\alpha(t)^2}\big).
	\end{align*}
	According to \eq{once_intgrtd} we have $|\dot\alpha(t)|\leq \sqrt{a_0}$ and $\alpha(t)\geq 1$ for 
	the solutions under consideration, so that the exponential factors yields finite $L^1$- and $BV$-norms
	of the conserved quantities at all times $t\geq 0$. Having fixed a separation constant $\lambda>0$, it is 
	also clear from the above expressions that these 
	norms can be made arbitrarily small, uniformly in time, by choosing $p_0$  small. 
	This is in agreement with the 1-d existence theory of Glimm \cite{gl} for small variation solutions:
	The conserved quantities remain small in $L^1(\RR)$ and $BV(\RR)$ for all times if sufficiently small 
	initially.
	
	Thus, Example \ref{instant_blowup} 
	highlights the following point about applying Glimm's theorem to concrete 1-dimensional 
	systems of conservation laws: 
	While the conserved quantities remain bounded, a physically relevant quantity 
	may well become unbounded (in our case, the velocity in the far-field).
	
	We note that the initial data in Example \ref{instant_blowup} are uniformly 
	within the regime of strict hyperbolicity, notwithstanding the fact that the data 
	decay to vacuum in the far-field (the characteristic speeds 
	at initial time are $u\equiv0$ and $u\pm c\equiv\pm \sqrt{\gamma\lambda/2}$). 
	Therefore, the instantaneous blowup in this example is not due to an initial loss of strict hyperbolicity 
	(the latter being a known source of possible blowup behavior \cite{sev}).
\end{remark}
%%%%%%%%%%%%%%%%%%%%%%%%%%%%%%%%%%%%%%%%%%%%
Summing up, we have that no accumulation takes place for affine flows in Case (A), but that 
unbounded velocities necessarily occur in the far-field. The lack of accumulation is reasonable 
since \eq{p_rho_reln}, with $\lambda>0$, gives an initial pressure profile which
decreases in the outward radial direction, and \eq{rad_aff_pressure} shows that the same 
holds at all later times as well. Thus, in Case (A), even if the velocity field is initially
converging, the negative pressure gradient ``wins'' and accelerates the fluid away from 
the center of motion at all times.

Also, the unbounded far-field behavior of the velocity is unsurprising: 
According to part (ii) of Observation 1, even when the fluid starts from rest,
the initial acceleration is unbounded in the far-field, and this results
in an unbounded velocity field at all later times.

%%%%%%%%%%%%%%%%%%%%%%%%%%%%%%
\subsection{Case (B): $\lambda<0$}\label{lambda_neg}
%%%%%%%%%%%%%%%%%%%%%%%%%%%%%%
The situation in this case is a little more involved since \eq{a_constr}
with $\lambda<0$ allows for both positive and negative values of 
the constant of integration $a$ in \eq{once_intgrtd}. 
We consider these two sub-cases separately, and we shall
see that accumulation occurs in certain cases.

%%%%%%%%%%%%%%%%%%%%%%%%%%%%%%
\subsubsection{{\em Case (B1): $\lambda<0$ and $a>0$}}\label{lambda_neg_a_pos}
%%%%%%%%%%%%%%%%%%%%%%%%%%%%%%
%%%%%%%%%%%%%%%%%%%
%	FIGURE 
%%%%%%%%%%%%%%%%%%%
\begin{figure}
	\centering
	\includegraphics[width=9cm,height=9cm]{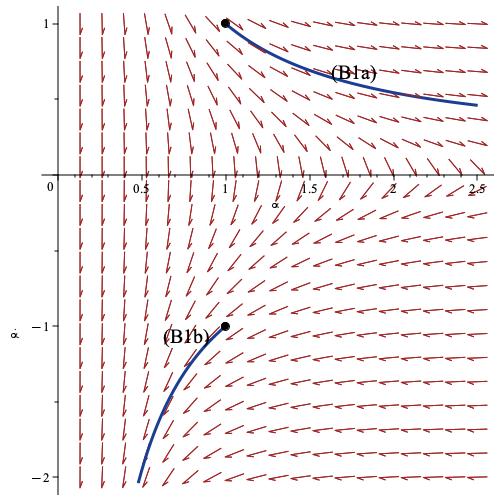}
	\caption{Case (B1): Two $(\alpha,\dot\alpha)$-trajectories for \eq{once_intgrtd} with 
	$\lambda<0<a$. Solid dots indicate initial values (with $\alpha(0)=1$)
	and arrows indicate direction of flow as
	time increases. The parameter values are $n=3$,
	$\gamma=\frac{5}{3}$, $\lambda=-\frac{15}{16}$.}\label{Figure_2}
\end{figure} 
%%%%%%%%%%%%%%%%%%%%%%%%%%%%%%%%%%%%%%%%%%%%

(See Figure \ref{Figure_2}.) In this case \eq{once_intgrtd} gives $(\dot\alpha)^2>a$,
so that all $(\alpha,\dot\alpha)$-trajectories under consideration start out and remain 
either above or below the horizontal strip $\{|\dot\alpha|\leq \sqrt a\}$ in the $(\alpha,\dot\alpha)$-plane. 
In particular, having $\dot\alpha(0)=0$ is not an option in this case: the initial velocity field is necessarily 
everywhere non-vanishing as well as unbounded in the far-field (according to \eq{rad_aff_speed}).

Defining $\alpha_*>0$ as in \eq{alpha_star}, solving \eq{once_intgrtd} for $\dot\alpha$, 
and integrating from time $t=0$ we now get
\beq\label{alpha_formula2}
	\int_1^{\alpha(t)} \textstyle\frac{\xi^{\frac{n(\gamma-1)}{2}}}
	{\sqrt{\xi^{n(\gamma-1)}+\alpha_*^{n(\gamma-1)}}}\,d\xi
	=\pm\sqrt{a}t
\eeq
(Note the difference of sign in the integrand compared to \eq{alpha_formula}.) 
Depending on $\sgn\dot\alpha(0)$ we have two further sub-cases (Figure \ref{Figure_2}):
\begin{itemize}
\item[(B1a)] If $\dot\alpha(0)>\sqrt a$, then the $(\alpha,\dot\alpha)$-trajectory is located in the 
	upper half of the $(\alpha,\dot\alpha)$-plane, and the `$+$' sign in \eq{alpha_formula2} 
	should be used. We have $\alpha(t)>1$ and $\dot\alpha(t)>\sqrt{a}$ for all $t>0$. 
	According to \eq{rad_aff_speed} we 
	obtain a radial flow in which all particles move away from the center of motion for 
	all times, the particle initially at radius $s$ moving with a speed which approaches 
	the asymptotic value $\sqrt{a}s$.
\item[(B1b)] If $\dot\alpha(0)<-\sqrt a$, then the flow has a different character. The 
	$(\alpha,\dot\alpha)$-trajectory is located in the 
	lower half of the $(\alpha,\dot\alpha)$-plane, and the  `$-$' sign in \eq{alpha_formula2} 
	should be used. The trajectory reaches $\alpha=0$ in a finite
	time $t_c=t_c(a)$ which, according to \eq{alpha_formula2}, is given by 
	\[t_c={\textstyle\frac{1}{\sqrt{a}}}\int_0^1 \textstyle\frac{\xi^{\frac{n(\gamma-1)}{2}}}
	{\sqrt{\xi^{n(\gamma-1)}+\alpha_*^{n(\gamma-1)}}}\,dx<\infty.\]
	From \eq{once_intgrtd} we have that $\dot\alpha(t)\downarrow-\infty$ as $t\uparrow t_c$. 
	We therefore obtain a cumulative flow in which all fluid particles accelerate toward 
	the origin and reach it at the common time $t_c$ with infinite speed. 
\end{itemize}
Thus, solutions are cumulative in case (B1b) but not in case (B1a). 
Since $\lambda<0$, we must prescribe an increasing initial 
pressure profile (cf.\ \eq{p_rho_reln}); i.e., the pressure gradient tends
to push fluid toward the center of motion in both cases. However,
in Case (B1a) the velocity field is diverging at all times and cumulative 
behavior is prevented: the pressure gradient ``looses'' and each particle 
moves off to infinity with asymptotically constant speed.
In Case (B1b) the initial velocity field is converging; coupled with the unbounded 
far-field velocity and the sign of the pressure gradient, it is not surprising that 
accumulation occurs in this case.

%%%%%%%%%%%%%%%%%%%%%%%%%%%%%%
\subsubsection{{\em Case (B2): $\lambda<0$ and $\frac{2\lambda}{n(\gamma-1)}\leq a<0$}}\label{lambda_neg_a_neg}
%%%%%%%%%%%%%%%%%%%%%%%%%%%%%%
(See Figure \ref{Figure_3}.)  Define $\alpha_*$ as in
\eq{alpha_star}; solving
\eq{once_intgrtd} for $\dot\alpha$ and integrating from time $t=0$ we now get
\beq\label{alpha_formula3}
	\int_1^{\alpha(t)} \textstyle\frac{\xi^{\frac{n(\gamma-1)}{2}}}{\sqrt{\alpha_*^{n(\gamma-1)}-\xi^{n(\gamma-1)}}}\,d\xi
	=\pm\sqrt{|a|}t
\eeq
%%%%%%%%%%%%%%%%%%%%%%%%%%%%%%
\begin{remark}\label{intgrbl2}
	The integrand on the left-hand side of \eq{alpha_formula3} is integrable at $\xi=\alpha_*-$.
\end{remark}
%%%%%%%%%%%%%%%%%%%%%%%%%%%%%%Depending on $\sgn\dot\alpha(0)$ we have:
\begin{itemize}
\item[(B2a)] If $\dot\alpha(0)>0$, then the $(\alpha,\dot\alpha)$-trajectory starts out in the 
	upper half of the $(\alpha,\dot\alpha)$-plane. Again, let $t_*=t_*(a)$ denote the (finite) time
	at which the trajectory reaches the $\alpha$-value $\alpha_*$ (where $\dot\alpha(t_*)=0$).
	For $t<t_*$  the `$+$' sign in \eq{alpha_formula3} 
	should be used, while the `$-$' sign should be used for $t>t_*$. 
	According to \eq{rad_aff} and \eq{rad_aff_speed} we therefore
	obtain a radial flow in which all particles initially slow down as they move away from the center of motion.
	Global stagnation occurs at time $t=t_*$, and then all particles accelerate toward the origin, reaching 
	it with infinite speed at the accumulation time $t_c$ given by
	\[t_c=t_*+{\textstyle\frac{1}{\sqrt{|a|}}}\int_0^{\alpha_*} \textstyle\frac{\xi^{\frac{n(\gamma-1)}{2}}}
	{\sqrt{\alpha_*^{n(\gamma-1)}-\xi^{n(\gamma-1)}}}\,d\xi,\]
	where $t_c<\infty$ according to Remark \ref{intgrbl2}.
\item[(B2b)] If $\dot\alpha(0)\leq0$, then the flow has the same character as in Case (B2a)  
	for times after $t_*$: Each particle accelerates toward the origin and reaches it at a
	common time $t_c$ given by 
	\[t_c={\textstyle\frac{1}{\sqrt{|a|}}}\int_0^1 \textstyle\frac{\xi^{\frac{n(\gamma-1)}{2}}}
	{\sqrt{\alpha_*^{n(\gamma-1)}-\xi^{n(\gamma-1)}}}\,d\xi.\]
\end{itemize}
Thus, all solutions are cumulative in Case (B2). Since $\lambda<0$
we have $\bar p'>0$, and 
the pressure profile tends to push the fluid toward the center of motion at all times.  
Although the initial velocity field is diverging in Case (B2a), the analysis above shows that 
the pressure gradient ``wins'' in this case, and all fluid particles end up accelerating toward the 
center of motion. 
The increasing pressure profile and the 
unbounded converging velocities in the far-field (subsequent to stagnation in Case (B2a)), 
render accumulation unsurprising in Case (B2).
(We note that, since $\bar p(s)$ is increasing in Case (B), an exterior vacuum is not an option 
as this would require $\bar p(s_0)=0$ at some finite $s_0>0$.)
%%%%%%%%%%%%%%%%%%%
%	FIGURE 
%%%%%%%%%%%%%%%%%%%
\begin{figure}
	\centering
	\includegraphics[width=9cm,height=9cm]{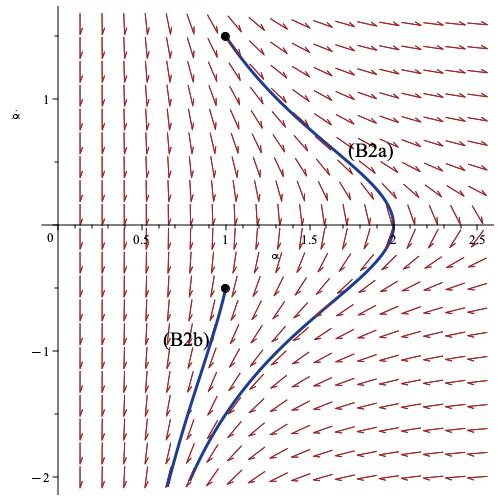}
	\caption{Case (B1): Two $(\alpha,\dot\alpha)$-trajectories for \eq{once_intgrtd} with 
	$\lambda<0$, $a<0$. Solid dots indicate initial values (with $\alpha(0)=1$)
	and arrows indicate direction of flow as
	time increases. The parameter values are $n=3$,
	$\gamma=\frac{5}{3}$, $\lambda=-3$.}\label{Figure_3}
\end{figure} 
%%%%%%%%%%%%%%%%%%%%%%%%%%%%%%%%%%%%%%%%%%%%

%
%
%%%%%%%%%%%%%%%%%%%%%%%%%%%%%%%%%%%%%%%%%
%\begin{remark}
%        For elastic bodies this type of solution {\em is} acceptable and yields 
%        cumulative behavior from physically acceptable initial data; see \cite{sid_2024}.
%\end{remark}
%%%%%%%%%%%%%%%%%%%%%%%%%%%%%%%%%%%%%%%%%
%

%%%%%%%%%%%%%%%%%%%%%%%%%%%%%%%%%%%%%%%%
%\begin{remark}
%        Keller \cite{kell} covers all cases considered above,
%        and also cases with $\gamma=1$ and/or $a=0$ (see also \cite{sed}).
%        He does not phrase his findings in terms of solutions to Cauchy problems but 
%        instead considers solutions as defined for all times $-\infty<t<\infty$. 
%        While discussing particle trajectories and indicating the cases where accumulation occur, 
%        he does not discuss the influence of the pressure gradient or the unbounded far-field behavior
%        of the velocity, the behavior of characteristics, 
%        nor the physicality of cumulative flows.
%\end{remark}
%%%%%%%%%%%%%%%%%%%%%%%%%%%%%%%%%%%%%%%%%

	%%%%%%%%%%%%%%%%%%%
	%	FIGURE 
	%%%%%%%%%%%%%%%%%%%
	\begin{figure}
		\centering
		\includegraphics[width=9cm,height=7cm]{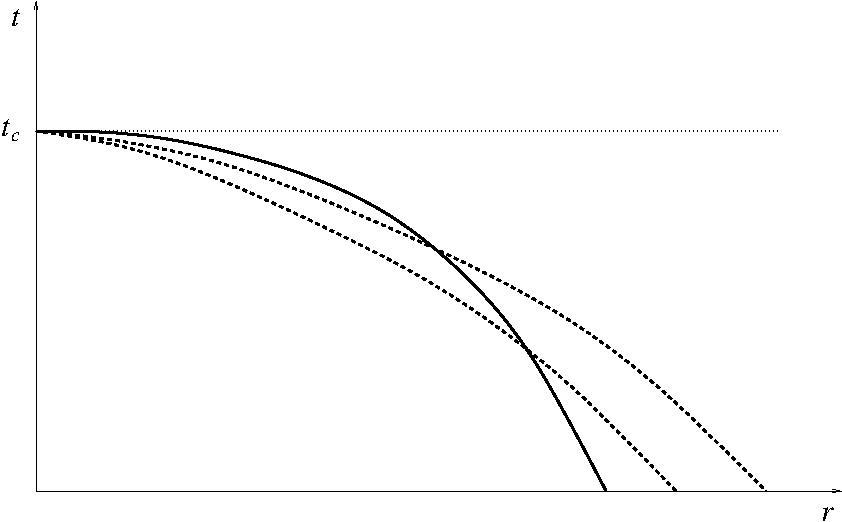}
		\caption{Schematic picture of a particle trajectory (solid) 
		and two 1-characteristics (dashed) in Case (B2) with initial 
		pressure profile $\bar p(s)=s^\kappa$ with $\kappa>0$.
		All particle paths and all 1-characteristics accumulate at the 
		center of motion at time $t_c$. Any particle trajectory is overtaken 
		by every 1-characteristic with a larger starting radius.}\label{Figure_4}
	\end{figure} 
%%%%%%%%%%%%%%%%%%%%%%%%%%%%%%%%%%%%%%%%%%%%

%%%%%%%%%%%%%%%%%%%%%%%%%%%%%%%%%%%
\begin{example}\label{ex_lam_neg}
	To generate a simple example of cumulative behavior in Case (B)
	we consider the situation with integration constant $a=0$ in \eq{once_intgrtd}
	and assume $n(\gamma-1)=2$. (For spatial dimensions 
	$n=1,\, 2,\, 3$, the last assumption amounts to $\gamma=3,\, 2,\, \frac 5 3$, respectively.) 
	With these choices, \eq{once_intgrtd}
	gives the ODE $\dot\alpha=\pm\sqrt{|\lambda|}/\alpha$. Choosing the `$-$' sign, and 
	recalling the assumption $\alpha(0)=1$, we obtain
	\beq\label{alph}
		\alpha(t)=\big(1-2\sqrt{|\lambda|}t\big)^\frac{1}{2}.
	\eeq
	Thus, the trajectory of the particle starting from initial position $s$ is given as
	\[r=R(t,s)=s\big(1-2\sqrt{|\lambda|}t\big)^\frac{1}{2},\]
	and it follows that the resulting flow suffers accumulation at the critical time
	\[t_c=\textstyle\frac{1}{2\sqrt{|\lambda|}}.\]
	We next consider the 1-characteristics $r=r(t)$ in the flow. To obtain an explicit expression 
	we need to prescribe a concrete (increasing) initial pressure profile $\bar p$. Letting
	\beq\label{p_bar_ex}
		\bar p(s)=s^\kappa\qquad\qquad (\kappa>0),
	\eeq
	we obtain from \eq{1_3_chars_f} (with the $-$ sign), together with \eq{alph} and the choice
	$n(\gamma-1)=2$, the following ODE for $F(t)=\frac{r(t)}{\alpha(t)}$:
	\[\sqrt\kappa\textstyle\frac{d}{dt}(\log F)=-\textstyle\frac{\sqrt{\gamma|\lambda|}}{\alpha^2}
	=-\textstyle\frac{\sqrt{\gamma|\lambda|}}{1-2\sqrt{|\lambda|}t}.\]
	Integrating in time and using $F(0)=r(0)$ give the explicit expression 
	for the 1-characteristics $r(t)=\alpha(t)F(t)$:
	\beq\label{1_char_lam_neg}
		r(t)=r(0)\big(1-2\sqrt{|\lambda|}t\big)^\beta \qquad\qquad
		\beta=\textstyle\frac{1}{2}\big(1+\sqrt{\textstyle\frac{\gamma}{\kappa}}\big).
	\eeq
	We therefore obtain an example of a cumulative flow in which all particle trajectories
	(viz.\ $r=\alpha(t)s$), as well as all 1-characteristics (viz.\ $r=\alpha(t)F(t)$),
	accumulate at the origin at the same critical time $t_c=1/2\sqrt{|\lambda|}$. See Figure \ref{Figure_4}. 
	(This provides a less singular example of the same behavior noted in the last part of 
	Example \ref{ex_lam_0}.) It follows from
	\eq{alph} and \eq{1_char_lam_neg} that, regardless of the value of $\kappa>0$ in \eq{p_bar_ex}, 
	any particle trajectory starting from position $s$ is overtaken by all 1-characteristics
	starting from positions $r(0)>s$ on its way toward the origin.
\end{example}
%%%%%%%%%%%%%%%%%%%%%%%%%%%%%%%%%%%

%%%%%%%%%%%%%%%%%%%%%%%%%%%%%%%%%%%
%%%%%%%%%%%%%%%%%%%%%%%%%%%%%%%%%%%
\section{1-dimensional non-affine flows exhibiting accumulation}\label{non_affine}
%%%%%%%%%%%%%%%%%%%%%%%%%%%%%%%%%%%
%%%%%%%%%%%%%%%%%%%%%%%%%%%%%%%%%%%
In this section we seek {\em non-affine} flows subject to the ansatz that the 
Jacobian of the motion $J(t,\by)=\det D_\by\bX(t,\by)$ is a function of
$t$ alone. This property holds for the radial affine motions \eq{rad_aff} 
considered earlier (where $J(t,\by)=\alpha(t)^n$, cf.\ \eq{J_sep} with $f(s)=s$).
We shall see that the $1$-dimensional Euler system admits certain non-affine 
motions of this type. These will provide additional 
examples of cumulative flows. However, they again suffer from unbounded 
velocity in the far-field. Thus, the main point of the following analysis
is to demonstrate that cumulative behavior can occur also in non-affine flows.

We start by considering a radial motion $\bX(t,\by)=R(t,s)\textstyle\frac{\by}{s}$ 
($s=|\by|$) in any dimension $n$. According to \eq{jac} its Jacobian is
\[J(t,s)=\big(\textstyle\frac{R}{s}\big)^{n-1}R_s.\]
The requirement that $J(t,s)\equiv J(t)$ implies that the motion $R(t,s)$
has the functional form
\beq\label{jac_t_motn}
	R(t,s)=[\psi(t)+J(t)s^n]^\frac{1}{n},
\eeq
where the functions $\psi(t)$, $J(t)$ are to be determined from 
the Lagrangian EOM \eq{eom_rad}. Observe that $\psi(t)\equiv 0$ 
yields affine motions; we are therefore only interested in cases
with non-trivial $\psi(t)$. Also, since we label particles by their 
initial (radial) position $s$, we have
\beq\label{ics}
	\psi(0)=0\qquad\text{and}\qquad J(0)=1.
\eeq
Using \eq{jac_t_motn} in the EOM \eq{eom_rad} yields the 
following equation (dropping the $t$-argument):
\beq\label{jac_t_eom}
	[\psi+Js^n]^{\frac{2}{n}-3}
	\left\{[\psi+Js^n][\ddot\psi+\ddot Js^n]
	+ (\textstyle\frac{1}{n}-1)[\dot\psi+\dot Js^n]^2\right\}
	=-nJ^{-\gamma}\textstyle\frac{\bar p'(s)}{s^{n-1}\bar\rho(s)}.
\eeq
We have not been able to identify non-affine solutions to \eq{jac_t_eom} 
(i.e., $\psi(t)\not\equiv 0$) in dimensions $n>1$. Restricting to the 1-dimensional case $n=1$
we obtain the equation
\beq\label{jac_t_eom_1_d}
	\ddot\psi+\ddot Js=-J^{-\gamma}\textstyle\frac{\bar p'(s)}{\bar\rho(s)}.
\eeq
There are two possibilities: either $\ddot\psi\equiv 0$ or $\ddot J\equiv 0$.
However, the former case gives $\psi(t)=kt$ for a constant $k$ (according to \eq{ics}${}_1$), 
while the functions $J(t)$, $\bar p(s)$, $\bar\rho(s)$ satisfy the equations 
\eq{alpha_eqn}-\eq{bar_p_bar_rho} for affine motions in 1-d  (with $J\equiv\alpha$).
The upshot is that the resulting motion $R(t,s)=\psi(t)+J(t)s$ in this case is simply 
a Galilean shift of an affine motion (cf.\ Remark \ref{Gal_inv}).
This, therefore, does not give anything essentially new.

On the other hand, the second possibility gives, according to \eq{ics}${}_2$,
$J(t)=1+bt$, where $b\neq0$ for non-trivial motions. The EOM \eq{jac_t_eom_1_d}
then reduces to 
\[\ddot\psi(t)=-\textstyle\frac{1}{(1+bt)^\gamma}\textstyle\frac{\bar p'(s)}{\bar\rho(s)},\]
so that, for  a separation constant $\lambda\in\RR$, we must have
\beq\label{relsns}
	\ddot\psi(t)=\textstyle\frac{\lambda}{(1+bt)^\gamma}\qquad\text{and}\qquad 
	\textstyle\frac{\bar p'(s)}{\bar\rho(s)}=-\lambda.
\eeq
The second of these equations define the initial density in terms of the initial 
pressure profile, while the first equation gives (recalling \eq{ics}${}_1$)
\beq\label{psi}
	\psi(t)=
	\left\{
	\begin{array}{ll}
		kt-\textstyle\frac{\lambda}{b^2}\log(1+bt) & \gamma=2\\ \\
		kt+\textstyle\frac{\lambda}{(\gamma-1)(\gamma-2)b^2}
		\left[(1+bt)^{2-\gamma}-1\right] & \gamma\neq2,
	\end{array}
	\right.
\eeq
for a constant $k$.
Without loss of generality (due to Galilean invariance, Remark \ref{Gal_inv}) the $kt$-term may be dropped,
and we obtain the motions
\beq\label{non_aff_motns}
	R(t,s)=
	\left\{
	\begin{array}{ll}
		-\textstyle\frac{\lambda}{b^2}\log(1+bt)+(1+bt)s & \gamma=2\\ \\
		\textstyle\frac{\lambda}{(\gamma-1)(\gamma-2)b^2}
		\left[(1+bt)^{2-\gamma}-1\right] +(1+bt)s& \gamma\neq2
	\end{array}
	\right.\qquad(b\neq0).
\eeq
For $\lambda=0$ we obtain affine flows already considered.
On the other hand, for $\lambda\neq0$ we obtain 
non-affine flows in which we are free to assign any monotone reference 
pressure $\bar p(s)\geq0$ subject to the constraint \eq{p_rho_reln}; the reference 
density is then defined by \eq{relsns}${}_2$. We note that the initial 
particle speed $\dot R(0,s)$ is unbounded as $s\to\pm\infty$ for any value of $\gamma>1$.

Assuming $\lambda\neq0$, the non-affine flows \eq{non_aff_motns} behave
differently depending on the sign of $b$. With $b>0$ we get globally defined 
flows in which all particle trajectories tend to $\pm\infty$ (for $s\gtrless0$, respectively) 
as $t\to\infty$. 
However, for $b<0$ we obtain singular behavior in finite time: For $\gamma\geq2$
all particle trajectories ``blow away'' to infinity as $t\uparrow t_c:=-\frac{1}{b}$, while
for $1<\gamma<2$ all particle trajectories accumulate at 
$r=r_c:=\frac{\lambda}{(\gamma-1)(\gamma-2)b^2}$ at time $t=t_c$.
This, then, provides the sought-for examples of 1-dimensional non-affine cumulative 
flows. 

We observe that \eq{non_aff_motns} is readily inverted to determine the inverse
``position-to-label'' map $s=S(t,r)$, so that the flow variables $\rho,p,u$ 
in the Eulerian frame are available explicitly (we omit the detailed expressions).

%%%%%%%%%%%%%%%%%%%%%%%%%%%%%%%%%%%
\section{Summary and Discussion}\label{disc}
%%%%%%%%%%%%%%%%%%%%%%%%%%%%%%%%%%%
%
%Note that accumulation, in the sense above (INCLUDE A DEFINITION!), is a highly singular event: {\em all}
%fluid particles are required to converge on the center of motion at a {\em common}
%critical time $t=t_c$. However, one can imagine a milder form of cumulative behavior, with
%distinct fluid particles arriving at the center of motion at distinct times, continuously building up 
%an increasingly stronger $\delta$-distribution in the density field. We are not aware of any example,
%``physical'' or not, of this type of behavior for the Euler equations.
%
%
%\vspace{1cm}
Above we derived the Lagrangian formulation of the Euler equations 
for compressible flow of an ideal and polytropic gas. We then specialized to
radial and to radial separated solutions, and we observed that the latter 
class is covered by the class of radial affine solutions. 
It was noted that the particular functional form of such flows imply 
that the velocity field is unbounded in the far field, possibly except at a 
single time of stagnation (possibly at initial time).
We then analyzed in some detail radial affine solutions with an emphasis 
on cases with accumulation where all fluid particles accumulate at the 
center of motion at a common time. 

We point out the following features of cumulative solutions in radial affine Euler flows
(in the absence of other forces than the internal thermodynamical pressure): 
\begin{enumerate}
	\item[(i)] There are of two distinct types of (affine) accumulation, driven either by 
	inertia or by pressure gradients. 
	\item[(ii)]  The first type occurs for separation constant $\lambda=0$, and 
	thus a vanishing pressure gradient, and requires a converging initial velocity field
	(cf.\ Section \ref{aff_lam_0}).
	The second type requires $\lambda<0$, and thus a positive radial pressure gradient, 
	and an initial velocity field that is not too strongly diverging (cf.\ Section \ref{aff_lam_not_0}).
	\item[(iii)]  All cases of accumulation are readily accounted for by physical considerations.
	\item[(iv)]  1-characteristics in cumulative flows necessarily reach the center of motion 
	no later than the time of accumulation; however, 1-characteristics may also accumulate 
	at the origin at the same time (cf.\ Examples \ref{ex_lam_0} and \ref{ex_lam_neg}) .
	\item[(v)]  Accumulation in radial affine flows is necessarily total, 
	i.e., {\em all} fluid trajectories $R(t,s)=\alpha(t)s$ approach $r=0$ as $t\uparrow t_c$,
	where $\alpha(t_c)=0$. Each fluid particle then experiences 
	unbounded growth in its pressure as $t\uparrow t_c$, cf.\ \eq{rad_aff_pressure_0} 
	and \eq{rad_aff_pressure}. The latter behavior appears necessary to accomplish 
	accumulation (i.e., preventing fluid particles to move away from the center of motion).
	\item[(vi)]  All cases of accumulation suffer from certain unphysical features:
	Either the initial velocity or the initial acceleration is unbounded in the far field;
	in either case unbounded velocities occur in the far-field at all later times
	(except, possibly, at a single time of stagnation).
	\item[(vii)] It is worth noting that all examples of accumulation analyzed above occur
	already in 1-dimensional flow; in particular, the multi-d effect of wave focusing does
	not play a role in generating these cumulative behaviors.
\end{enumerate}
These points raise the question of whether accumulation can occur without initially unbounded 
velocity or acceleration.
As a first step it is natural consider far-field
modifications (for $r\geq \bar r>0$, say) of the initial data for the cumulative solutions 
studied above. By this we mean changing the data for $r\geq \bar r$,
either continuously or discontinuously at $r=\bar r$, to data 
with bounded velocity and acceleration in the far-field. 
However, we have not been able to build an Euler flow in which 
the modified flow leaves intact the particle trajectories starting from 
positions $r\in[0,\bar r]$. Instead, changes in the data for $r\geq \bar r$ 
introduce perturbations that overtake the particle trajectories 
which should, supposedly, accumulate. 

More precisely, if the modified flow remains continuous,
then the introduced modifications for $r\geq \bar r$ will necessarily influence, via perturbations 
traveling with 1-characteristic speeds, the inner particle trajectories. 
Next, an admissible 1-shock inserted at $r=\bar r$ will overtake 1-characteristics, and thus 
necessarily also particle trajectories, on its way toward $r=0$. Inserting a 3-shock 
$r=\bar r$ will also introduce perturbations propagating with 1-characteristic speeds,
again influencing the inner particle trajectories. (In the latter case, the perturbations 
are propagated along 1-characteristics ``originating'' from the inside of the 3-shock.)  
Apparently, the only way to avoid having the far-field modification of the data 
pollute the inner accumulating flow, is to have the perturbed and unperturbed 
parts of the flow connected by a contact discontinuity (i.e., a particle trajectory). 

However, there is an immediate issue with this last scenario: 
Since the pressure in the accumulating inner part of the flow 
becomes unbounded, and the pressure does not change across a contact, we 
must provide an outer flow in which the innermost fluid particle (propagating along
the contact) undergoes unbounded growth in pressure. However, we are not aware of 
any example of an Euler flow with this property (except, of course, the radial affine
cumulative solutions that we seek to modify). In particular, to the best of our knowledge,
pressure blowup along a particle trajectory does not occur in any of the 
known cases of self-similar Euler flows \cites{gud,hun_60,jt3,laz,mrrs1,jls}.

While we recognize that the preceding arguments do not prove that 
modification of the far-field in known accumulating flows must necessarily 
lead to non-accumulating flows, we are not optimistic about this approach.
Instead, it seems that a different type of mechanism would be required to generate 
truly acceptable cases of cumulative behavior. If such flows exist it is reasonable to 
expect that they would involve wave-focusing in dimensions $n\geq 2$.
In conclusion, it remains an open question if accumulation in Euler flow
of an ideal gas can be attained with initially bounded velocity and acceleration fields.

%BIBLIOGRAPHY

\end{document}